\documentclass[a4paper,11pt,reqno]{amsart}

\usepackage{hyperref}
\hypersetup{backref,colorlinks=true,allcolors=black}
\usepackage[foot]{amsaddr}
\usepackage{mathrsfs}
\usepackage{enumitem}
\usepackage{dsfont}
\usepackage{mathtools}
\usepackage{amssymb}
\usepackage{amsthm}
\usepackage{cases}

\usepackage{xcolor}
\definecolor{airforceblue}{rgb}{0.36, 0.54, 0.66}

\newtheorem{prop}{Proposition}
\newtheorem{lemma}{Lemma}
\newtheorem{corollary}{Corollary}
\newtheorem{theorem}{Theorem}

\theoremstyle{definition}

\newtheorem{example}{Example}
\newtheorem{remark}{Remark}

\newcommand{\ignore}[1]{}
\newcommand{\R}{\mathbb{R}}
\newcommand{\N}{\mathbb{N}}

\newcommand{\norm}[1]{\left\lVert #1 \right\rVert}
\newcommand{\abs}[1]{\left\vert #1 \right\rvert}
\newcommand{\E}[1]{\mathbb{E}{\left[ #1\right]}}

\newcommand{\V}[1]{\mathrm{Var}{\left( #1\right)}}

\definecolor{bleu}{rgb}{0.08, 0.38, 0.74}

\DeclareMathOperator{\Tr}{Tr}
\DeclarePairedDelimiter\autobracket{(}{)}
\DeclarePairedDelimiterX\car[1]\lvert\rvert{\ifblank{#1}}
\newcommand{\brac}[1]{\autobracket*{#1}}
\newcommand{\inner}[1]{\left\langle #1 \right\rangle}

\theoremstyle{definition}

\makeatletter
\@addtoreset{proofcase}{theorem}
\makeatother

\theoremstyle{definition}

\makeatletter
\@addtoreset{proofstep}{theorem}
\makeatother

\begin{document}
	%============================================================================================%
	\title{Functional Gaussian approximations on Hilbert-Poisson spaces}
	
	\thanks{S. Bourguin was supported in part by the Simons Foundation
		grant 635136}

	\author{Solesne Bourguin$^1$}
	\address{$^1$Boston University, Department of Mathematics and Statistics,  111 Cummington Mall, Boston, MA 02215, USA}
	\email{bourguin@math.bu.edu}
	\author{Simon Campese$^2$}
	\address{$^2$Hamburg University of Technology, Institute of Mathematics, Am Schwarzenberg-Campus 3, 21073 Hamburg, Germany}
	\email{simon.campese@tuhh.de}
	\author{Thanh Dang$^1$}
	\email{ycloud77@bu.edu}

	\begin{abstract}
		
		We develop a functional Stein-Malliavin method in a non-diffusive Poissonian setting, thus obtaining a) quantitative central limit theorems for approximation of arbitrary non-degenerate Gaussian random elements taking values in a separable Hilbert space and b) fourth moment bounds for approximating sequences with finite chaos expansion.
		% As in the diffusive case treated earlier by the first two authors, all known finite-dimensional fourth moment results for Gaussian approximation in a non-diffusive context are included as special cases.
		Our results rely on an infinite-dimensional version of Stein's method of exchangeable pairs combined with the so-called Gamma calculus. Two applications are included: Brownian approximation of Poisson processes in Besov-Liouville spaces and a functional limit theorem for an edge-counting statistic of a random geometric graph.
	\end{abstract}

	\bibliographystyle{amsalpha}
	
	\maketitle      
	\section{Introduction}
	The now classical Stein-Malliavin method, a combination of Stein's method with Malliavin calculus, has been very successful in deriving quantitative central limit theorems for non-linear approximation. Since its inception by Nourdin and Peccati in 2013 (see~\cite{nourdin-peccati:2009:steins-method-wiener}), it has formed a vivid community which developed the theory further and applied it to numerous situations. An excellent exposition of the basic method is available in the monograph~\cite{nourdin-peccati:2012:normal-approximations-malliavin} ,  while I. Nourdin keeps a rather exhaustive and continuously updated list of references on the webpage \texttt{https://sites.google.com/site/malliavinstein}. From a theoretical point of view, one of the main remaining challenges is an adaptation of the method to the infinite-dimensional setting, with quantitative approximation of Gaussian processes as main application. For random elements taking values in a Hilbert space, and in a diffusive context, this has recently been achieved by~\cite{bourguin-campese:2020:approximation-hilbert-valued-gaussians}. In this work, we provide the natural analogue in the non-diffusive context of Poisson spaces. More specifically, let $X$ be a square-integrable measurable transformation of a Poisson process and $Z$ be a Gaussian process, both taking values in some separable Hilbert space $K$. Informally, our main results (Theorems~\ref{theorem_fourmomentHilbert} and \ref{theorem_contractionestimate} on page \pageref{theorem_fourmomentHilbert}) provide bounds on a probabilistic distance between $X$ and $Z$ (metrizing convergence in law) in terms of the first four strong moments of $X$ or alternatively in terms of so called contractions. From these bounds, one can directly deduce quantitative and functional central limit theorems for convergence towards a Gaussian process, as well as an infinite-dimensional version of the Fourth Moment Theorem, which says that for a sequence of $K$-valued multiple Poisson-integrals, convergence of the second and fourth moments implies convergence towards a Gaussian process.
	
	It is noteworthy to observe that while the analogous diffusive statements in~\cite{bourguin-campese:2020:approximation-hilbert-valued-gaussians} look similar to our non-diffusive ones, their proofs are rather different, for the same reason as in the finite-dimensional case: no chain rule is available in the non-diffusive case, which renders the usual integration by parts argument unfeasible. Instead, one can construct an appropriate exchangeable pair and then apply a Taylor argument in order to control the term resulting from an application of Stein's method. Compared to the finite-dimensional setting, several technical issues arise which require the use of Hilbert-space techniques. A commonality with the diffusive statements is, however, that our main results subsume all known finite-dimensional Malliavin-Stein bounds in a Poissonian context as special cases (see Remark~\ref{rmk:1} on page~\pageref{rmk:1} for details).
	
	In order to illustrate our results, we provide two applications: The first one concerns the classical approximation of a Brownian motion by a normalized Poisson process with growing intensity $\lambda$. A natural class of Hilbert spaces accommodating the sample paths of both processes are the so-called Besov-Liouville spaces. In~\cite{coutin-decreusefond:2013:steins-method-brownian}, the authors showed that convergence takes place at rate $\lambda^{-1/2}$ (as in the classical one-dimensional case). To prove this, they first transferred both processes isometrically $\ell^2(\mathbb{N})$ and then had to go through rather tedious calculations. In contrast to this, our bounds yield the same result in just a few lines, and no isometry is necessary.
	As a second application we illustrate, using an edge counting statistic of a random graph, how known one-dimensional central limit theorem can be made functional with very little additional effort.
	
	Besides the already mentioned reference~\cite{bourguin-campese:2020:approximation-hilbert-valued-gaussians}, the work~\cite{coutin-decreusefond:2013:steins-method-brownian}  is also concerned with quantitative functional approximation in a Malliavin-Stein context. As already mentioned, the authors use a different approach which crucially depends on isometrically mapping all random elements to $\ell^2(\mathbb{N})$. In applications, the need to explicitly evaluate such an isometry can be seen as a drawback. Also, our setting seems to be more general and does not rely on ad-hoc arguments depending on the Gaussian process at hand. Other related references proving functional central limit theorems using Malliavin-Stein techniques are ~%
	\cite{kasprzak:2017:multivariate-functional-approximations%
		,kasprzak:2020:functional-approximations-via%
		,dobler-kasprzak:2021:steins-method-exchangeable%
		,dobler-kasprzak-peccati:2019:functional-convergence-u-processes%
	}.
	
	The rest of this paper is organized as follows. In Section~\ref{Section_prelim} we introduce the necessary preliminaries, followed by the main results in Section~\ref{sec:stat-main-results}. The proofs are given in Section~\ref{sec:proof-main-results} which is followed by the two aforementioned applications in Section~\ref{sec:applications}. An appendix contains several technical lemmas required for the proofs.
	
	\section{Preliminaries}
	\label{Section_prelim}
	\subsection{Probability on Hilbert spaces}
	\hfill\\
	\indent Let $K$ be a real separable Hilbert space, $\mathcal{B}(K)$ the Borel $\sigma$-algebra of $K$ and $\brac{\Omega,\mathcal{F},P}$ a complete probability space. A $K$-valued random variable $X$ is a measurable map from $\brac{\Omega,\mathcal{F}}$ to $\brac{K,\mathcal{B}(K)}$. Such random variables are characterized by the property that for any continuous linear functional $\phi\in K^*$, the function $\phi(X):\Omega\to\R$ is a real-valued random variable. As usual, the distribution or law of $X$ is the push-forward probability measure $P\circ X^{-1}$ on $\brac{K,\mathcal{B}(K)}$. The set of all $K$-valued random variables is a a vector space over the field of real numbers. If the Lebesgue integral $\E{\norm{X}_K}=\int_\Omega \norm{X}_KdP$ exists and is finite, then the Bochner integral $\int_\Omega XdP$ exists in $K$ and is called the expectation of $X$. Slightly abusing notation, we denote this integral by $\E{X}$ as well, and it can always inferred from the context whether $\E{\cdot}$ refers to Lebesgue or Bochner integration with respect to $P$. For $p\geq 1$, $L^p\brac{\Omega,P}$ denotes the Banach space of all equivalence classes (under almost sure equality) of $K$-valued random variables $X$ with finite $p$-th moment, i.e., such that
	\begin{align*}
		\norm{X}_{L^p\brac{\Omega,P}}=\E{\norm{X}^p_K}^{1/p}<\infty.
	\end{align*}
	Note that for all $X\in L^p\brac{\Omega,P}$, the Bochner integral $\E{X}$ exists. In the case $X\in L^2\brac{\Omega,P}$, the covariance operator $S:K\to K$ of $X$ is defined by
	\begin{align*}
		Su=\E{\inner{X,u}_K X}.
	\end{align*}
	$S$ is a positive, self-adjoint trace-class operator that verifies the identity
	\begin{align*}
		\Tr S= \E{\norm{X}^2_K}.
	\end{align*}
	We denote by $S_1(K)$ the Banach space of all trace-class operators on $K$, equipped with norm $\norm{T}_{S_1(K)}=\Tr \abs{T}$, where $\abs{T}=\sqrt{T T^{\ast}}$ and $T^{\ast}$ denotes the adjoint of $T$. The subspace of Hilbert-Schmidt operators on $K$ is denoted by $\operatorname{HS}(K)$, its inner product and norm by $\inner{\cdot,\cdot}_{\operatorname{HS}(K)},\norm{\cdot}_{\operatorname{HS}(K)}$ respectively. Recall that
	\begin{align*}
		\norm{\cdot}_{\operatorname{op}}\leq \norm{\cdot}_{\operatorname{HS}(K)}\leq \norm{\cdot}_{S_1(K)},
	\end{align*}
	where $\norm{\cdot}_{\operatorname{op}}$ denotes the operator norm.

	\subsection{Gaussian measures and Stein's method}
	\hfill\\
	\indent In this section, we introduce Gaussian measures, the associated abstract Wiener spaces and Stein characterization of Gaussian measures. The theory will be presented within a general Banach space setting. Note that at the end of this section and beyond that, we will assume any target Gaussian measure under consideration is defined on a Hilbert space such as $K$ above. Standard references for Gaussian measures and abstract Wiener spaces are the monographs \cite{bogachev:1998:gaussian-measures,kuo:1975:gaussian-measures-banach}, while Stein's method for Gaussian measures has been developed by Shih in \cite{shih:2011:steins-method-infinite-dimensional} (see also Barbour's earlier work~\cite{barbour:1990:steins-method-diffusion} for the special case of Brownian motion).
	\subsubsection{Abstract Wiener spaces}
	Let $H$ be a real separable Hilbert space equipped with inner product $\inner{\cdot,\cdot}_H$ and $\norm{\cdot}$ be a norm on $H$ weaker than $\norm{\cdot}_H$. Denote $B$ the Banach space obtained via completion of $H$ with respect to $\norm{\cdot}$ and $i$ the canonical embedding of $H$ into $B$. The triple $(i,H,B)$ defines an abstract Wiener space and has first been introduced by Gross in~\cite{gross:1967:abstract-wiener-spaces}. We identify $B^*$ as a dense subspace of $H^{\ast}$ under the adjoint $i^*$ of $i$, so that we have the continuous embeddings $B^*\subseteq H\subseteq B$, where, as usual, $H$ is identified with its dual $H^{\ast}$.  All of this can be summarized via the diagram
	\begin{align*}
		B^*\xrightarrow{i^*} H^*=H \xrightarrow{i} B.
	\end{align*}
	The abstract Wiener measure $p$ on $B$ is characterized as the Borel measure on $B$ satisfying
	\begin{align*}
		\int_B \exp\brac{{i\inner{x,\eta}}_{B,B^*}} p(dx)=\exp\brac{-\frac{\norm{\eta}^2_H}{2}},
	\end{align*}
	for any $\eta\in B^*$. 
	\subsubsection{Gaussian measures}
	Let $B$ be a separable Banach space, with $\mathcal{B}(B)$ its
	Borel $\sigma$-algebra. A Gaussian measure $\mu$ is a
	probability measure on $(B,\mathcal{B}(B))$ such that every
	linear functional $x\in B^*$, considered as a (real-valued) random variable on $(B,\mathcal{B}(B),\mu)$, has a Gaussian distribution on
	$(\mathbb{R},\mathcal{B}(\mathbb{R}))$. Such a Gaussian measure is called centered and/or non-degenerate, if these properties hold
	for the distributions of every $x\in B^*$.
	
	We can see that every abstract Wiener measure is a Gaussian measure, and conversely, for every Gaussian measure $\mu$ on $B$, there exists a Hilbert space $H$ such that $(i,H,B)$ forms an abstract Wiener space. The space $H$ is known as the Cameron Martin space.
	
	\subsubsection{Stein characterization of Gaussian measures}
	Let $B$ be a real separable Banach space with norm
	$\norm{\cdot}$. Let $Z$ be a $B$-valued random variable which induces a centered Gaussian measure $\mu_Z$ on $B$ and let $(i,H,B)$ be the associated abstract Wiener space. By $\{P_t:t\geq 0\}$ we denote the Ornstein-Uhlenbeck semi-group of $Z$. It has the Mehler representation
	\begin{align*}
		P_tf(x)=\int_B f\brac{e^{-t}x+\sqrt{1-e^{-2t}}y}\mu_Z(dy),
	\end{align*}
	provided such an integral exists. In \cite[Theorem 3.1]{shih:2011:steins-method-infinite-dimensional}, Shih proved the following Stein lemma for abstract Wiener measures.
	\begin{theorem}
		\label{theorem_steinlemma}
		Let $X$ be a $B$-valued random variable with distribution $\mu_X$.
		\begin{enumerate}[label=\roman*)]
			\item If $B$ is finite-dimensional, then $\mu_X=\mu_Z$ if and only if
			\begin{align}
				\label{Stein_characterization}
				\E{\inner{X,\nabla f(X)}_{B,B^*}-\Delta_G f(X)}=0
			\end{align}
			for any twice-differentiable function $f$ on $B$ such that $\E{\norm{\nabla^2 f(Z)}_{S_1{(H)}}}<\infty$.
			\item If $B$ is infinite-dimensional, then $\mu_X=\mu_Z$ if and only if \eqref{Stein_characterization} holds for every twice $H$-differentiable function $f$ on $B$ such that $\nabla f(x)\in B^*$ for every $x\in B$, \\$\E{\norm{\nabla^2 f(Z)}_{S_1(H)}}<\infty$ and $\E{\norm{\nabla f(Z)}_{B^*}^2}<\infty$. 
		\end{enumerate}
	\end{theorem}
	The notion of $H$-derivative which is also known as Fréchet derivative along $H$ and appears in Theorem
	\ref{theorem_steinlemma} was introduced by Gross in \cite{gross:1967:potential-theory-hilbert}, and
	we briefly recall it here for the sake of self-containedness. A function $f:U\to W$  from an open set $U$ of $B$ into a Banach space $W$ is said to be $H$-differentiable at $x\in U$ if the map $\phi(h)=f(x+h),h\in H$, regarded as a function defined in a neighborhood of the origin of $H$ is Fréchet-differentiable at $0$. The $H$-derivative of $f$ at $x$ in the direction $h\in H$ is denoted by $\inner{\nabla f(x), h}_H$. The $k$-th order $H$-derivatives of $f$ at $x$ can then be constructed inductively and are denoted by $\nabla^k f(x)$, provided they exist. If $f$ is scalar-valued, $\nabla f(x)\in H^* \simeq H$ and $\nabla^2 f(x)$ is a bounded linear operator from $H$ to $H^*$ for every $x\in U$. The notation $\inner{\nabla^2 f(x)h,k}_H$ or $\nabla^2 f(x)(h,k)$ will stand for the action of the linear form $\nabla^2 f(x)(h,\cdot)$ on $k$.
	
	\noindent If $\nabla^2 f(x)$ is a trace-class operator on $H$, the Gross Laplacian $\Delta_{G}f(x)$ of $f$ at $x$ is defined as $\Delta_G f(x)=\Tr_H(\nabla^2 f(x))$.
	
	\subsubsection{Stein's equation}
	In view of Theorem \ref{Stein_characterization}, the associated Stein equation is given by
	\begin{align*}
		\inner{x,\nabla g(x)}_{B,B^*}-\Delta_G g(x)=h(x)-\E{h(Z)}
	\end{align*}
	for $x\in B$, where $h$ belongs to a suitable class of test functions. Recall that we assume $K$ to be a separable Hilbert space. From this point forward, we will let $B=K$ and assume our test functions belong to $C^3_b(K)$, the class of real-valued functions on $K$ that have bounded Fréchet  derivatives up to order three. This space is equipped with the norm \begin{align*}\norm{h}_{C^3_b(K)}=\sup_{j=1,2,3}\sup_{x\in K}\norm{D^jh(x)}_{K^{\otimes j}}.\end{align*}
	Using standard semigroup techniques, the first two authors of this work showed in  \cite{bourguin-campese:2020:approximation-hilbert-valued-gaussians} that there is a
	solution $g_h(x)$ for every test function $h(x)$ and that
	$g_h\in C^3_b(K)$ when $h\in C^3_b(K)$. Specifically,
	\cite[Lemma 2.4]{bourguin-campese:2020:approximation-hilbert-valued-gaussians} provides the estimates
	\begin{align}
		\label{estimate_steinequaitonsol}
		\sup_{x\in K}\norm{D^jg_h(x)}_{K^{\otimes j}}\leq \frac{1}{j}\norm{h}_{C^j_b(K)}
	\end{align}
	and
	\begin{align*}
		\norm{g_h}_{C^3_b(K)}\leq \norm{h}_{C^3_b(K)}.
	\end{align*}
	Thus, using the probability distance
	\begin{align*}
		d_3(X_1,X_2)=\sup_{\substack{h\in C^3_b(K)\\\norm{h}_{C^3_b(K)}\leq 1}}\abs{\E{h(X_1)-h(X_2)}},
	\end{align*}
	Stein's equation implies that
	\begin{align*}
		d_3(X,Z)= \sup_{\substack{h\in C^3_b(K)\\\norm{h}_{C^3_b(K)}\leq 1}}\abs{\E{\Delta_G g_h(X)-\inner{X,Dg_h(X)}_K}}.
	\end{align*}
	\begin{remark}
Since the focus of our paper is to obtain functional fourth moment theorems on separable Hilbert spaces, we choose to work with a rather strong notion of probability distance that is $d_3(\cdot,\cdot)$. In fact, proofs of our main results (see Section \ref{sec:proof-main-results}) only require bounds on second and third derivatives of $g_h$. Then, per \eqref{estimate_steinequaitonsol}, the results in Section \ref{sec:stat-main-results} still hold true if $d_3(\cdot,\cdot)$ is replaced by the probability distance
			\begin{align*}
				\widetilde{d}(X_1,X_2)=\sup_{h\in\mathcal{A}}\abs{\E{h(X_1)-h(X_2)}}
			\end{align*}
			such that
			$\mathcal{A}=\{h:K\to\R\text{ and }\sup_{x\in K}\norm{D^2h(x)}_{K^{\otimes 2}}\vee \sup_{x\in K}\norm{D^3h(x)}_{K^{\otimes 3}}\leq 1\}$.
			Alternatively, one can refer to \cite[Theorem 4.9v]{shih:2011:steins-method-infinite-dimensional} which states that if $h$ is a Lipschitz function on $K$, then $
			\sup_{x\in K}\norm{D^2g_h(x)}_{\operatorname{HS}(K)}\leq \norm{h}_{\operatorname{Lip}}$.
			Based on this fact, one can replace $d_3(\cdot,\cdot)$ in the results in Section \ref{sec:stat-main-results} with 
			\begin{align*}
				\bar{d}(X_1,X_2)=\sup_{h\in\mathcal{B}}\abs{\E{h(X_1)-h(X_2)}}
			\end{align*}
			such that $\mathcal{B}=\{h:K\to\R\text{ and }
			\sup_{x\in K}\norm{Dh(x)}_{K}\vee \sup_{x\in K}\norm{D^3h(x)}_{K^{\otimes 3}}\leq 1\}$.
		
		A more interesting question here is whether one can
                obtain the results in Section
                \ref{sec:stat-main-results} with the Wasserstein
                distance, which is defined by
                \begin{align*}
				d_{\operatorname{Wass}}(X_1,X_2)=\sup_{\norm{h}_{\operatorname{Lip}(K)}\leq 1}\abs{\E{h(X_1)-h(X_2)}}.
			\end{align*}
			The answer seems to be negative. In our proofs
                        in Section \ref{sec:proof-main-results}, we
                        will apply Taylor's theorem, which then
                        requires a bound on
                        $\norm{D^3g_h}_{\operatorname{op}}$. However,
                        even in the finite-dimensional setting of multivariate normal approximations, the author of \cite{raivc2004multivariate} has constructed a counterexample of a Lipschitz test function $h$ such that $\norm{\partial^3g_h}_{\operatorname{op}}$ is not bounded. 
	\end{remark}
	\subsection{Dirichlet structure}        
	\hfill\\
	\indent This section contains an overview of Dirichlet
	structures, which is the framework we will be working within
	alongside Stein's method. We start by recalling the definition
	and properties of a Dirichlet structure on $L^2(\Omega;\R)$
	(full details can be found in the monographs \cite{bakry-gentil-ledoux:2014:analysis-geometry-markov, bouleau-hirsch:1991:dirichlet-forms-analysis}) before focusing on an extension to $L^2(\Omega;K)$. Given a probability space $\brac{\Omega,\mathcal{F},P}$, a Dirichlet structure $\brac{\mathbb{D},\mathcal{E}}$ on $L^2(\Omega;\R)$ with the associated carré du champ operator $\Gamma$ consists of a Dirichlet domain $\mathbb{D}$, which is a dense subset of $L^2(\Omega;\R)$ and a carré du champ operator $\Gamma:\mathbb{D}\times \mathbb{D}\to L^1(\Omega,\R)$ characterized by the following properties.
	\begin{itemize}
		\item[-]  $\Gamma$ is bilinear, symmetric ($\Gamma(F,G)=\Gamma(G,F)$) and positive ($\Gamma(F,F)\geq 0$).
		\item[-] the induced positive linear form $F\to \mathcal{E}(F,F)$, where $\mathcal{E}(F,G)=\frac{1}{2}\E{\Gamma(F,G)}$, is closed in $L^2(\Omega;\R)$, i.e., $\mathbb{D}$ is complete when equipped with the norm
		\begin{align*}
			\norm{\cdot}^2_{\mathbb{D}}=\norm{\cdot}^2_{L^2(\Omega;\R)}+\mathcal{E}(\cdot). 
		\end{align*}
	\end{itemize}
	\begin{remark}
		We do not assume that $\Gamma$ satisfies the
		so-called diffusion property -- see \cite[Definition
		3.1.3]{bakry-gentil-ledoux:2014:analysis-geometry-markov} -- as opposed to what is being done in \cite{bourguin-campese:2020:approximation-hilbert-valued-gaussians}.
	\end{remark}
	Here and in the following, $\E{\cdot}$ denotes the expectation on $\brac{\Omega,\mathcal{F}}$ with respect to $P$. The linear form $\mathcal{E}$ is known as a Dirichlet form and for brevity we write $\mathcal{E}(F)$ for $\mathcal{E}(F,F)$. Every Dirichlet form gives rise to a strongly continuous semigroup $\left\{P_t  \right\}_{t\geq 0}$ on $L^2(\Omega;\R)$ and an associated symmetric Markov generator $-L$, defined on a dense subset $\operatorname{dom}(-L)\subseteq \mathbb{D}$. There are two important relations between $\Gamma$ and $L$, the first being the integration by part formula
	\begin{align*}
		\E{\Gamma(F,G)}=-\E{FLG}=-\E{GLF},
	\end{align*}
	which is valid for $F,G\in\mathbb{D}$. The second relation is
	\begin{align*}
		\Gamma(F,G)=\frac{1}{2}\brac{L(FG)-GLF-FLG},
	\end{align*}
	which holds for all $F,G\in \operatorname{dom}(L)$ such that
	$FG\in \operatorname{dom}(L)$. If $-L$ is diagonalizable with spectrum $\N_{0}$ (the set of natural numbers plus $0$) and $F_q$ is an eigenfunction corresponding to the eigenvalue $q$,
	then $-L F_q=qF_q$. We can define a pseudo-inverse $-L^{-1}$
	by $-L^{-1}F_q=\frac{1}{q}F_q$ when $q\neq 0$ and $0$
	otherwise. The definition of $-L$ and $-L^{-1}$ for a general
	$F=\sum_{q\in \N_{0}}F_q$ follows naturally via
	linearity. Alternatively, $L$ can be defined as the generator of the heat semigroup $\left\{P_t  \right\}_{t\geq 0}$ (on $\operatorname{dom}(L)$) which satisfies
	\begin{align*}
		\partial_t P_t=LP_t=P_tL.
	\end{align*}
	
	\noindent Next we present what is meant by a Dirichlet structure on $L^2(\Omega;K)$. Let us adopt the notations $\widetilde{\mathbb{D}}, \widetilde{\Gamma},\widetilde{L}, \widetilde{P_t}$ for the Dirichlet domain, Dirichlet form, carré du champ operator, generator and semigroup associated with elements in $L^2(\Omega;\R)$. Meanwhile, ${\mathbb{D}}, {\Gamma},{L}, {P_t}$ are reserved for the counterpart objects associated with elements in $L^2(\Omega;K)$. Given a separable Hilbert space $K$, one has that $L^2(\Omega;K)$ is isomorphic to $L^2(\Omega;\R)\otimes K$. The Dirichlet structure on $L^2(\Omega;\R)$ can therefore be extended to $L^2(\Omega;K)$ via a tensorization procedure. Let $\N_{0}$ be the spectrum of $-\widetilde{L}$ and $\{ k_i\}_{i\in\N}$ an orthonormal basis of $K$. $\mathcal{A}$ will be the set of all functions $X$ taking the form
	\begin{align*}
		X=\sum_{q,i\in I}F_{q,i}\otimes k_i
	\end{align*}
	such that $I\subseteq \N^2$ is a finite set and $F_{q,i}\in \operatorname{ker}\brac{-\widetilde{L}+qI}$. Assuming another element $Y=\sum_{p,j\in J}G_{p,j}\otimes k_j$ in $\mathcal{A}$,  we can define $L,\Gamma,P_t,\mathcal{E}$ for $t\geq 0$ via
	\begin{equation*}
		\begin{cases}
			\displaystyle   LX=L\sum_{q,i\in I}F_{q,i}\otimes k_i=\sum_{q,i\in I}\brac{\widetilde{L}F_{q,i}}\otimes k_i\\
			\displaystyle   P_tX=P_t\sum_{q,i\in I}F_{q,i}\otimes k_i=\sum_{q,i\in I}\brac{\widetilde{P_t}F_{q,i}}\otimes k_i\\
			\displaystyle   \Gamma(X,Y)=\frac{1}{2}\sum_{q,i\in I}\sum_{p,j\in J}\widetilde{\Gamma}(F_{q,i},F_{p,j})\otimes\brac{k_i\otimes k_j+k_j\otimes k_i}
		\end{cases}
	\end{equation*}
	and
	\begin{align*}
		\mathcal{E}(X,Y)=\E{\Tr \Gamma(X,Y)}.
	\end{align*}
	In the last line, we identify $\Gamma(X,Y)$ as an element of $L^2(\Omega;\R)\otimes K\otimes K\simeq L^2(\Omega,\mathcal{L}(K,K))$ via the action
	\begin{align*}
		\Gamma(X,Y)u=\frac{1}{2}\sum_{q,i\in I}\sum_{p,j\in J}\widetilde{\Gamma}(F_{q,i},F_{p,j})\otimes\brac{\inner{k_i,u}_K\otimes k_j+\inner{k_j,u}_K\otimes k_i}. 
	\end{align*}
	\begin{remark}
			The definitions above are independent of the choice of basis of $K$. We include here a brief explanation for the operator $L$ for the sake of completeness. First, the definition of $L$ is equivalent to
			\begin{align*}
				LX=\sum_{i\in N}\brac{\widetilde{L}\inner{X,k_i}_K}k_i.
			\end{align*}
			Let $\{e_j \}_{j\in \N}$ be another orthonormal basis of $K$ and define
			\begin{align*}
				L_0X=\sum_{j\in\N}\brac{\widetilde{L}\inner{X,e_j}_K}e_j.
			\end{align*}
			Then by using $e_j=\sum_{n\in \N}\inner{e_j,k_n}_Kk_n$ such that $\sum_{n\in \N}\inner{e_j,k_n}^2_K=1$, and also the identity $\inner{u,v}_K=\sum_{i\in\N}\inner{u,k_i}_K\inner{v,k_i}_K$, one can deduce that $L_0X=LX$. This confirms that the definition of $L$ is indeed basis-independent.  
	\end{remark}

	Since $\mathcal{A}$ is clearly dense in $L^2(\Omega;K)$, the operators above can be extended to appropriate domains in $L^2(\Omega;K)$. This has been verified in \cite[Proposition 2.5 and Theorem 2.6]{bourguin-campese:2020:approximation-hilbert-valued-gaussians} (excluding the diffusion identity), which we restate below for the reader's convenience.

	\begin{prop}[Proposition 2.5 in \cite{bourguin-campese:2020:approximation-hilbert-valued-gaussians}]
		The operators $L$ $L^{-1}$, $\mathcal{E}$ and $\Gamma$ can be extended to $\operatorname{dom}(L)$, $\operatorname{dom}(L^{-1})$ and $\operatorname{dom}(\Gamma)=\operatorname{dom}(\mathcal{E})=\mathbb{D}\times\mathbb{D}$, respectively, given by
		\begin{equation*}
			\operatorname{dom}(L)=\Big\{X\in L^2(\Omega;K):\sum_{q\in \N_{0}} q^2 \widetilde{J}_{q}\brac{\norm{X}^2_K}<\infty \Big\},
		\end{equation*}
		$\operatorname{dom}(L^{-1})=L^2(\Omega;K)$ and
		\begin{align*}
			\mathbb{D}&=\Big\{X\in L^2(\Omega;K):\sum_{q\in \N_{0}} q \widetilde{J}_{q}\brac{\norm{X}^2_K}<\infty \Big\},
		\end{align*}
		where $\widetilde{J}_{q}(\cdot)$ denotes the projection
		onto
		$\operatorname{ker}\brac{\widetilde{L}+qI}\subseteq
		L^2(\Omega;\R)$. In particular, one has \begin{align*}
			\mathcal{A}\subseteq \operatorname{dom}(L)\subseteq \mathbb{D}\subseteq \operatorname{dom}(L^{-1})=L^2(\Omega;K),\end{align*} and all inclusions are dense. 
	\end{prop}
	
	\begin{theorem}[Theorem 2.6 in \cite{bourguin-campese:2020:approximation-hilbert-valued-gaussians}]
		For a Dirichlet structure $(\mathbb{D},\Gamma)$ on $L^2(\Omega;K)$, the following is true. 
		\begin{enumerate}[label=(\roman*)]
			\item $\Gamma$ is bilinear, almost surely positive, symmetric and self-adjoint with respect to $\inner{\cdot,\cdot}_K$.
			\item The Dirichlet domain $\mathbb{D}$ equipped with the norm 
			\begin{align*}
				\norm{X}_{\mathbb{D}}^2=\norm{X}_{L^2(\Omega;K)}+\norm{\Gamma(X,X)}_{L^1(\Omega;S_1)}
			\end{align*}
			is complete, so that $\Gamma$ is closed.
			\item The generator $-L$ acting on $L^2(\Omega;K)$ is positive, symmetric, densely defined and has the same spectrum as $-\widetilde{L}$.
			\item There is a compact pseudo-inverse $L^{-1}$ of $L$ such that 
			\begin{align*}
				LL^{-1}X=X-\E{X}
			\end{align*}
			for all $X\in L^2(\Omega;K)$, where the expression on the right is a Bochner integral. 
			\item The integration by parts formula
			\begin{align*}
				\E{\Tr\Gamma(X,Y)}=-\E{\inner{LX,Y}_K}=-\E{\inner{X,LY}_K}
			\end{align*}
			is satisfied for all $X,Y\in \operatorname{dom}(-L)$.
			\item The generators $\Gamma,L,\widetilde{L}$ are related via
			\begin{align}
				\Tr\Gamma(X,Y)=\frac{1}{2}\brac{\widetilde{L}\inner{X,Y}_K-\inner{LX,Y}_K-\inner{X,LY}_K}
			\end{align}
			for all $X,Y\in \operatorname{dom}(-L)$.
			\item The identity 
			\begin{align*}
				\inner{\Gamma(X,Y)u,v}_K=\frac{1}{2}\brac{\widetilde{\Gamma}\brac{\inner{X,u}_K,\inner{Y,v}_K}+\widetilde{\Gamma}\brac{\inner{Y,u}_K,\inner{X,v}_K}},
			\end{align*}
			is valid for all $X,Y\in\mathbb{D}$ and $u,v\in K$. 
		\end{enumerate}
	\end{theorem}
	\subsection{Analysis on Poisson space}\label{subsection_Poissonspace}      \hfill\\
	\indent  So far we have been working with a general
	probability space. In this section we will get more specific
	and describe the Poisson space on which most of our objects of
	interest are defined. We direct the reader to the references
	\cite{last-penrose:2018:lectures-poisson-process,nualart-nualart:2018:introduction-malliavin-calculus} for an extensive treatment of this topic. Let $(\mathcal{Z},\mathscr{L},\mu)$ be a measure space such that $\mu $ is $\sigma$-finite. A Poisson random measure $\eta$ on $(\mathcal{Z},\mathscr{L})$ with control measure $\mu$ is a family of distributions defined on some probability space $(\Omega,\mathcal{F},P)$ that satisfies
	\begin{itemize}
		\item[-] $\eta(B)$ is a Poisson distribution on $\Omega$ with mean $\mu(B)$,
		\item[-] $\eta(B_1), \eta(B_2)$ are independent when $B_1\cap B_2=\emptyset$. 
	\end{itemize}
	If such a Poisson random measure exists, the associated
	probability space $(\Omega,\mathcal{F},P)$ is called a Poisson
	space. Next, let $\widehat{\eta}$ be the compensated Poisson
	random measure, that is $\widehat{\eta}(B)=\eta(B)-\mu(B)$,
	whenever $\mu(B)$ is finite. Denote $L^2_s(\mu^q)$ the set of
	all symmetric functions in $L^2(\mu^q)$. For $f\in
	L^2_s(\mu^q)$, $I_q^\eta(f)$ denotes a multiple (Wiener-It\^o)
	integral of order $q$. Unless we are simultaneously dealing
	with two different Poisson random measures, $I_q(\cdot)$ will be understood as an integral with respect to $\widehat{\eta}$. Multiple integrals have the following isometry property: for
	any integers $q,p \geq 1$,
	\begin{align*}
		\E{I_{q}(f) I_{p}(g)} = \mathds{1}_{\left\{ q=p \right\}} q! \langle \tilde{f},\tilde{g}\rangle _{L^2(\mu^q)},
	\end{align*}
	where $\tilde{f} $ denotes the symmetrization of $f$, and we
	recall that $I_q(f) = I_q ( \tilde{f})$. The contraction of
	two kernels $f\in L^2_s(\mu^q)$ and $g\in L^2_s(\mu^p)$,
	denoted by $f\star^l_r g$ for $0\leq l\leq r\leq q\wedge p$, is obtained by identifying $r$ variables and then integrating $l$ of those:
	\begin{align*}
		&f\star^l_r g\brac{y_{1},\ldots,y_{r-l},y_{r-l+1},\ldots,y_{q-l},z_1,\ldots,z_{p-r}}\\
		&=\int_{\mathcal{Z}^l}f(x_1,\ldots,x_l,y_{1},\ldots,y_{r-l},y_{r-l+1},\ldots,y_{q-l})g(x_1,\ldots,x_l,y_{1},\ldots,y_{r-l},z_1,\ldots,z_{p-r})\\&\hspace*{31em} d\mu\brac{x_1,\ldots,x_l}
	\end{align*}
	provided the integral exists in
	$L^2(\mu^{q+p-r-l})$. Contractions are central objects for
	analysis on Poisson space as they appear in the product
	formula for multiple integrals. There are two ways of stating
	this product formula on Poisson space: \cite[Proposition 6.1]{last:2016:stochastic-analysis-poisson} and \cite[Lemma 2.4]{dobler-peccati:2018:fourth-moment-theorem}, each having different assumptions. We will state both below.
	\begin{lemma}[Proposition 6.1 in \cite{last:2016:stochastic-analysis-poisson}]
		Let $f\in L^2_s(\mu^q),g\in L^2_s(\mu^p)$ and assume that $f\star^l_r g\in L^2(\mu^{q+p-r-l})$. Then,
		\begin{align}
			\label{productformulaPoisson}
			I_q(f)I_p(g)=\sum_{r=0}^{q\wedge p} r!{q\choose r}{p\choose r}\sum_{l=0}^r{r\choose l}I_{q+p-r-l}(f\star^l_r g).
		\end{align}
	\end{lemma}
	
	\begin{lemma}[Lemma 2.4 in \cite{dobler-peccati:2018:fourth-moment-theorem}]
		Let $f\in L^2_s(\mu^q),g\in L^2_s(\mu^p)$ and assume that $F=I_q(f),G=I_p(g)\in L^4(P)$. Then
		\begin{align*}
			FG=\sum_{k=1}^{q\wedge p-1}\widetilde{J}_k(FG)+I_{q+p}(f\widetilde{\otimes}g). 
		\end{align*}
	\end{lemma}
	The collection of all multiple integrals of order
	$q$ form the so-called Poisson chaos of order $q$ in $L^2(\Omega;\R)$, which is denoted by $\mathcal{H}_q$. Since $\E{I_q(f)I_p(g)}=0$ for $q\neq p$, we have the orthogonal decomposition
	\begin{align*}
		L^2(\Omega,\mathcal{F},P)=\bigoplus_{q=1}^\infty \mathcal{H}_q.
	\end{align*}
	Similarly as what we did for Dirichlet structures, we define
	$\mathcal{H}_q(K)$ ($K$-valued Poisson chaos of order $q$) as
	the closure of $\mathcal{H}_q\otimes K$ in $L^2(\Omega, K)$. Then, 
	\begin{align*}
		L^2(\Omega;K)=\bigoplus_{q=1}^\infty \mathcal{H}_q(K).
	\end{align*}
	Consequently, every $X\in L^2(\Omega, K)$ can be decomposed as
	\begin{align*}
		X=\sum_{q\in \N_{0}}F_q=\sum_{\substack{i\in\N,\\q\in \N_{0}}}\inner{F_q,k_i}_K k_i=\sum_{\substack{i\in\N,\\q\in \N_{0}}}F_{q,i}k_i,
	\end{align*}
	where $F_q\in\mathcal{H}_q(K)$, $F_{q,i}\in\mathcal{H}_q$ with
	$F_{q,i}=I_q(f_{q,i})$ for some $f_{q,i}\in L^2_s(\mu^q)$.
	% \\~\\
	% We can now introduce a Dirichlet structure $(\widetilde{\mathbb{D}},\widetilde{\Gamma})$ on our Poisson space $\brac{\Omega,P}$ just like in the previous section. In particular, for $q\in \N_{0}$, $\mathcal{H}_q(K)$ are eigenspaces of $-\widetilde{L}$. If $X$ is a random variables with the spectral decomposition $X=\sum_{q\in \N_{0}}F_q$ then  
	% \begin{align*}
		%       -\widetilde{L} X=-\widetilde{L}\brac{\sum_{q\in \N_{0}}F_q}=\sum_{q\in \N_{0}}qF_q.
		% \end{align*}
	% Finally, we have a Dirichlet structure $\brac{\mathbb{D},\Gamma}$ on $\Omega\otimes K$.
	\subsection{An exchangeable pair on Poisson space}
	
	\label{constructionofthepair}
	Another tool that we employ alongside Stein's
          method is an exchangeable pair on the Poisson space. The construction of this crucial exchangeable pair is done in \cite{dobler-vidotto-zheng:2018:fourth-moment-theorems} (see also \cite{zhengrademacher,Zhengexchangeablegaussian} for analogous construction on Rademacher and Gaussian spaces).
	
	Let $f$ be a measurable function $\N\to \R$ such that $\E{f(\eta)}<\infty$. Based on \cite[Chapter 20]{last-penrose:2018:lectures-poisson-process}, the semi-group $\{P_t\}_{t\geq 0}$ on the Poisson space admits the representation
		\begin{align*}
			P_tf(\eta)=\mathbb{E} \left[f\brac{\eta_{e^{-t}}+\hat{\eta}_t} \mid \eta\right]
		\end{align*}
		for $t\geq 0$. Here $\eta_{e^{-t}}$ is the $e^{-t}$-thinning of $\eta$. $\hat{\eta}_t$ is another Poisson random measure with control $\brac{1-e^{-t}}\mu(\cdot)$ and is independent from $(\eta,\eta_{e^{-t}})$. The above representation is also known as Mehler's formula on the Poisson space.  
	
	The paper \cite{dobler-vidotto-zheng:2018:fourth-moment-theorems} contains the following important result regarding the process $\eta^t=\eta_{e^{-t}}+\hat{\eta}_t$.
		\begin{lemma}[Lemma 3.1 in \cite{dobler-vidotto-zheng:2018:fourth-moment-theorems}] For each $t\geq 0$,  $\brac{\eta,\eta^t}$ is an exchangeable pair of Poisson random measures. 
		\end{lemma}
		As a result, for any kernel $g\in L^2_s(\mu^p)$, the pair $\brac{I^\eta_p(g),I^{\eta^t}_p(g)}$ is also exchangeable. Further relations between $I^\eta_p(g)$ and $I^{\eta^t}_p(g)$ are stated in Lemma \ref{lemmaexchangeablepairDVZ} and can be proved with the Mehler's formula above. 
	\section{Statement of main results}
	\label{sec:stat-main-results}
	In what follows, let $K$ be a separable Hilbert space with
	orthonormal basis $\{k_i \}_{i\in\N}$, and let $X$ denote a $K$-valued centered
	random variable in $L^2 \left( \Omega;K \right)$ with finite chaos decomposition 
	\begin{equation}
		\label{chaosdecompofX}
		X = \sum_{1\leq q\leq N} F_q,
	\end{equation}
	where each $F_q$ belongs to the $q$-th $K$-valued Poisson chaos. Furthermore, assume that $X$ has covariance
	operator $S$, which in turn decomposes as
	\begin{equation*}
		S = \sum_{1\leq q\leq N} S_q,
	\end{equation*}
	where, for each $1 \leq q \leq N$, $S_q$ is the covariance operator of
	$F_q$. Finally, we will denote by $f_{q,i} \in \mathfrak{H}^{\otimes
		q}$ the kernel of $F_{q,i} = \left\langle F_q,k_i \right\rangle_K =
	I_q \left( f_{q,i} \right)$.

	Our first main result provides a quantitative bound on the distance
	between the law of $X$ and a centered $K$-valued Gaussian random
	variable $Z$ in terms of the first four moments of $X$.
	% Our first result is a four moment estimate. Consistent with the setup in Section \ref{Section_prelim},  $K$ is a separable Hilbert space with an orthonormal basis $\{k_i \}_{i\in\N}$. Let $X$ be a $K$-valued Poisson random variable such that $X$ belongs to $L^2(\Omega;K)$ and has covariance operator $S$. Further assume $X$ has a finite Wiener-Ito decomposition $\sum_{1\leq q\leq N}F_{q}$, for which $F_q$ is in $\mathfrak{H}(K)$ and is equipped with the covariance operator $S_q$. Also $f_{q,i}\in\mathfrak{H}^{\otimes q}$ is the the kernel in $F_{q,i}=\inner{F_q,k_i}_K=I_q\brac{f_{q,i}}$. 
	
	\begin{theorem}
		\label{theorem_fourmomentHilbert}
		Assume $X$ is a $K$-valued random variable as
		described above such that for every $1\leq q\leq N$, $F_q$ has finite fourth moment, i.e.,
			$\E{\norm{F_q}_K^4}<\infty$. Then, letting $Z$ be a
		centered Gaussian random variable on $K$ with covariance operator $S'$, the following estimate holds
		\begin{align*}
			d_3(X,Z)
			&\leq\frac{1}{2}\norm{S-S'}_{\operatorname{HS}}
			\\ &\quad+\sum_{1\leq
				q\leq
				N}\frac{2q-1}{4q}\sqrt{\E{\norm{F_q}^4_K}-\E{\norm{F_q}^2_K}^2-2\norm{S_q}^2_{\operatorname{HS}}}\\
			&\quad +\sum_{1\leq p\neq q\leq
				N}\frac{p+q-1}{4p}\sqrt{\E{\norm{F_p}^2_K\norm{F_q}^2_K}-\E{\norm{F_p}^2_K}\E{\norm{F_q}^2_K}}\\
			& +\frac{N}{2}\sqrt{\max_{1\leq p\leq N}\E{\norm{F_p}^2_K}}\sum_{1\leq q\leq N} \sqrt{4q-3}\sqrt{{\norm{F_q}_K^4-\E{\norm{F_q}^2_K}^2-2\norm{S_q}^2_{\operatorname{HS}}}}
			\\ &\leq
			\frac{1}{2}\norm{S-S'}_{\operatorname{HS}}
			\\ & \quad +\brac{\frac{N(2N-1)}{4}+\frac{N}{2}\sqrt{(4N-3)\max_{1\leq p\leq N}\E{\norm{F_p}^2_K}}}
			\\ &\hspace{10em}  \sqrt{\E{\norm{X}^4_K}-\E{\norm{X}^2_K}^2-2\norm{S}^2_{\operatorname{HS}}}.
		\end{align*}
		% This estimate can be further reduced to
		%                 \begin{align*}
			%                         & d_3(X,Z) \leq
			%                    \frac{1}{2}\norm{S-S'}_{\operatorname{HS}}
			%                   \\ & \quad +\brac{\frac{N(2N-1)}{4}+\sqrt{2^{3N-1}N(4N-3)\E{\norm{X}^2_K}}}
			% \sqrt{\E{\norm{X}^4_K}-\E{\norm{X}^2_K}^2-2\norm{S}^2_{\operatorname{HS}}}.
			%                 \end{align*}
	\end{theorem}
	\begin{remark}
		\label{rmk:1}
		Note that Theorem \ref{theorem_fourmomentHilbert} is an infinite-dimensional version of 
		the fourth moment theorems on the Poisson space
		obtained in \cite[Theorem 1.2, Theorem 1.7]{dobler-vidotto-zheng:2018:fourth-moment-theorems} and
		\cite[Theorem 1.3]{dobler-peccati:2018:fourth-moment-theorem}. In particular, the aforementioned results are special cases of Theorem
		\ref{theorem_fourmomentHilbert} obtained by setting
		$K=\R^d$ for a positive integer $d$.
		
		Furthermore, whenever $d \geq 2$, there is an important difference between our result and
		that of \cite[Theorem 1.7]{dobler-vidotto-zheng:2018:fourth-moment-theorems} which lies in the fact that the linear
		combination of the first four moments appearing in the statement of
		\cite[Theorem 1.7]{dobler-vidotto-zheng:2018:fourth-moment-theorems} has a $1/4$ exponent when our results only
		involves a square root, which provides a much better convergence
		rate. This is due to the fact that the estimate in \cite{dobler-vidotto-zheng:2018:fourth-moment-theorems} is
		expressed in term of basis elements of $\R^d$ (see \cite[(4.3) and
		Lemma 4.1]{dobler-vidotto-zheng:2018:fourth-moment-theorems}), while our estimate is derived after performing a
		summation over all the basis elements of $K$. For more specific
		details, we refer the reader to the calculations involved in the proof
		of Theorem \ref{theorem_fourmomentHilbert} in Section \ref{Section_prooffourmomentHilbert}. 
	\end{remark}
	\begin{remark}
		Observe that Theorem \ref{theorem_fourmomentHilbert} can be
		viewed as a Poissonian counterpart of \cite[Theorem
		3.10]{bourguin-campese:2020:approximation-hilbert-valued-gaussians} in the context of a non-diffusive chaos
		structure. The fact that we are working with a non-diffusive
		structure (where no chain rule is available for the Gamma
		calculus introduced in Section \ref{Section_prelim}) forces
		us to use different techniques in order to obtain the above
		quantitative bounds than the ones used in \cite{bourguin-campese:2020:approximation-hilbert-valued-gaussians}, making
		these results comparable in nature, but very different in
		their methodologies of proof.
	\end{remark}

	\begin{remark}
		\label{remark_whyfinitefourmoment}
		We explain here why we assume in Theorem \ref{theorem_fourmomentHilbert} (and consequently in all other theoretical results of this paper, which follow from this theorem) that for every $1\leq q\leq N$, $\E{\norm{F_q}_K^4}<\infty$. In the proof of Theorem \ref{theorem_fourmomentHilbert} in Section \ref{sec:proof-main-results}, we will apply at several locations an important estimate that is Lemma \ref{lemma_DVZlemma2.2}. This lemma is originally established in \cite{dobler-peccati:2018:fourth-moment-theorem,dobler-vidotto-zheng:2018:fourth-moment-theorems} and is obtained via the product formula \eqref{productformulaPoisson}, which requires that $F_{q,i}=\inner{F_q,k_i}_K$ has finite fourth moment. This requirement is satisfied in the event $\E{\norm{F_q}_K^4}<\infty$, since 
			\begin{align*}
				\E{\norm{F_q}_K^4}=\E{\norm{\sum_{i\in\N}F_{q,i}k_i }_K^4}=\sum_{i,j\in\N}\E{F_{q,i}^2F_{q,j}^2}\geq \sum_{i\in\N}\E{F_{q,i}^4}. 
		\end{align*}
	\end{remark}
	
	\begin{remark}
		\label{remark_sufficientcondition4thmoment}
		Let us mention some helpful criteria for the purpose of verifying finite fourth moment of $F=I_q(f)$ for some $f\in L^2_s(\mu^q)\otimes K$. Based on \cite[Remark 1.2b]{dobler-peccati:2018:fourth-moment-theorem}, a sufficient condition is that $\norm{f}_K$ is bounded and the support of $\norm{f}_K$ is contained in a rectangle of the type $C\times C\ldots\times C$ such that $\mu(C)<\infty$, where we recall $\mu$ is the control of our Poisson random measure. This detail can be easily verified via the product formula \eqref{productformulaPoisson}. Poisson multiple integrals with this type of kernels include most U-statistics that are relevant to geometric applications (see \cite{lachieze-rey-peccati:2013:fine-gaussian-fluctuations, lachieze-rey-peccati:2013:fine-gaussian-fluctuations1,reitzner-schulte:2013:central-limit-theorems} and the references therein). 
		
		Another way to verify that $F$ has finite fourth moment is via restricted hypercontractivity on the Poisson space recently discovered in \cite{nourdin2020restricted}. Specifically, under technical assumptions in their Theorem 1.4, one can bound the fourth moment of $F$ by its variance.      
	\end{remark}
	Whenever $X$ belongs to a single chaos, we can reformulate Theorem
	\ref{theorem_fourmomentHilbert} in a more compact form:
	
	\begin{corollary}[Quantitative Fourth Moment Theorem]
		\label{corollary_fourmomentHilbert}
		Let the notation and setup of Theorem
		\ref{theorem_fourmomentHilbert} prevail. When $X$ belongs to a single
		chaos, i.e., $X\in \mathcal{H}_q(K)$ for some $q \geq
		1$, one has
		\begin{align*}
			d_3(X,Z)\leq &\frac{1}{2}\norm{S-S'}_{\operatorname{HS}}\\&+\brac{\frac{q(2q-1)}{4}+\frac{q}{2}\sqrt{(4q-3)\E{\norm{X}^2_K}} }\sqrt{ \E{\norm{X}_K^4}-\E{\norm{X}^2_K}^2-2\norm{S}^2_{\operatorname{HS}}}.
		\end{align*}
	\end{corollary}
	
	As $d_3$ metrizes convergence in law, the above corollary in particular shows that within a single non-diffusive chaos, convergence of the second and fourth strong moments implies convergence towards a (Hilbert-valued) Gaussian.
	~\\~\\
	A particularly useful formulation of the above moment bounds for
	applications uses contraction operators acting on the kernels of the
	multiple integrals appearing in the chaos decomposition representation
	of $X$ given in \eqref{chaosdecompofX}. Contractions, which are the
	analytic quantities defined in Section \ref{Section_prelim}, allow for
	much simpler computation compared to dealing directly with the first
	four moments. Some examples of previous works that use contraction norms to obtain quantitative limit theorem for Poisson random variables include \cite{lachieze-rey-peccati:2013:fine-gaussian-fluctuations%
		,lachieze-rey-peccati:2013:fine-gaussian-fluctuations1%
		,reitzner-schulte:2013:central-limit-theorems%
	}.
	
	Our second main result is the following contraction bound.
	\begin{theorem}
		\label{theorem_contractionestimate}
		Let the notation and setup of Theorem
		\ref{theorem_fourmomentHilbert} prevail. Moreover, let $\mathfrak{H}=L^2(\mathcal{Z},\mu)$ where $\mathcal{Z}$ is the $\sigma$-fnite measure space described in Subsection \ref{subsection_Poissonspace}. Then it holds that
		\begin{align*}
			d_3(X,Z)\leq \brac{\frac{N(2N-1)}{4}+\frac{N}{2}\sqrt{(4N-3)\max_{1\leq p\leq N}\E{\norm{F_p}^2_K}}}\sqrt{\beta}+\frac{1}{2}\norm{S-S'}_{\operatorname{HS}},
		\end{align*}
		where the quantity $\beta$ is given (in terms of
		contraction norms) by
		\begin{align*}
			\beta&=\sum_{\substack{1 \leq p,q \leq N\\q\neq
					p}}a_{p,q}(p\wedge q)\norm{f_q \star^{q\wedge
					p}_{q\wedge p}
				f_p}^2_{\mathfrak{H}^{\otimes\abs{q-p}} \otimes
				K^{\otimes 2}}\\
			&\quad +\sum_{1 \leq p,q \leq N}\sum_{r=1}^{q\wedge p -1}b_{p,q}(r) \norm{f_{q} \star^r_r f_{p}}^2_{\mathfrak{H}^{\otimes (q+p-2r)}\otimes K^{\otimes 2}}\\
			& \quad +\sum_{1 \leq p,q \leq N}\sum_{(r,s,l,m)\in I}c_{p,q,l,m}(r,s)\norm{{f_{q}{\star}^l_rf_{p}}}_{\mathfrak{H}^{\otimes (q+p-r-l)}\otimes K^{\otimes 2}}\norm{{f_{q}
					{\star}^m_sf_{p}}}_{\mathfrak{H}^{\otimes (q+p-r-l)}\otimes K^{\otimes 2}}.
		\end{align*}
		Here, the combinatorial coefficients are given by
		\begin{equation*}
			\begin{cases}
				\displaystyle a_{p,q}(r) = p!q! \binom{q}{r}\binom{p}{r} + r!^2 \binom{q}{r}^2
				\binom{p}{r}^2 \abs{p-q}! \\
				\displaystyle b_{p,q}(r) = p!q! \binom{q}{r}\binom{p}{r} \\
				\displaystyle c_{p,q,l,m}(r,s) = r!s! \binom{q}{r}\binom{q}{s}\binom{p}{r}\binom{p}{s}\binom{r}{l}\binom{s}{m}(p+q-r-l)!
			\end{cases},
		\end{equation*}
		and the index set $I$ is defined by
		\begin{align*}I=\{(r,s,l,m)\in \N^4\colon & 0\leq
			r,s\leq q\wedge p,\ 0\leq l\leq r,\ 0\leq m\leq s,\\
			&
			r+l=s+m,\ (r,s,l,m)\notin \{(0,0,0,0),(q\wedge p,q\wedge p,q\wedge p,q\wedge p) \}\}.\end{align*}
	\end{theorem}
	\begin{example}
		\label{remark_contraction_order2}
		If $X$ is a sum of elements of the first two chaoses, i.e.,
		$X=I_1(f_1)+I_2(f_2)$, Theorem
		\ref{theorem_contractionestimate} requires the
		contraction norms $\norm{f_{1}\star^1_1
			f_{2}}_{\mathfrak{H}\otimes K^{\otimes 2}}$,
		$\norm{f_{2}\star^1_1 f_2}_{\mathfrak{H}^{\otimes
				2}\otimes K^{\otimes 2}}$, $\norm{f_{1}\star^0_1
			f_{2}}_{\mathfrak{H}^{\otimes 2}\otimes K^{\otimes
				2}}$, $\norm{f_{1}\star^0_1
			f_{2}}_{\mathfrak{H}^{\otimes 2}\otimes K^{\otimes
				2}}$, $\norm{f_{2}\star^0_2
			f_{2}}_{\mathfrak{H}^{\otimes 2}\otimes K^{\otimes
				2}}$, $\norm{f_{2}\star^1_2
			f_{2}}_{\mathfrak{H}\otimes K^{\otimes 2}}$ and $\norm{f_{1}\star^0_1 f_{1}}_{\mathfrak{H}\otimes K^{\otimes 2}}$ to converge to $0$ to get
		convergence towards a Gaussian law.
	\end{example}
	\begin{example}
		Let $\mu$ be a $\sigma$-finite measure on some
		measure space. By setting $K=\R, \mathfrak{H}=L^2(\mu)$ and $X=I_p(f)$
		for some $p\geq 2$ in Theorem \ref{theorem_contractionestimate}, we
		get a result comparable to \cite[Theorem 5.1]{peccati-sole-taqqu-ea:2010:steins-method-normal} and \cite[Theorem
		2]{peccati-taqqu:2008:central-limit-theorems}. For instance, whenever $X=I_2(f)$, Theorem
		\ref{theorem_contractionestimate} and \cite[Example 5.2]{peccati-sole-taqqu-ea:2010:steins-method-normal} both
		state that normal convergence happens if $\norm{f\star^1_1
			f}_{L^2(\mu^2)}$, $\norm{f}_{L^4(\mu^2)}$ and $\norm{f\star^1_2
			f}_{L^2(\mu)}$ converge to $0$, keeping in mind that
		$\norm{f}^2_{L^4(\mu^2)}=\norm{f\star^0_2 f}_{L^2(\mu^2)}$, and $\norm{f\star^0_1 f}_{L^2(\mu^3)}=\norm{f\star^1_2 f}_{L^2(\mu)}$.

		Another example is \cite[Example 5.3]{peccati-sole-taqqu-ea:2010:steins-method-normal}, which states that $X=I_3(g)$ converges to a Gaussian distribution if
		$\norm{g}^2_{L^4(\mu^3)}$, $\norm{g\star^1_1
			g}_{L^2(\mu^4)}$, $\norm{g\star^1_2
			g}_{L^2(\mu^3)}$, $\norm{g\star^1_3 g}_{L^2(\mu^2)}$
		and $\norm{g\star^2_3 g}_{L^2(\mu)}$ all converge to
		$0$, which is the same condition suggested in Theorem \ref{theorem_contractionestimate}.

		Further, we would like to mention \cite{eichelsbacher-thale:2014:new-berry-esseen-bounds%
			,lachieze-rey-peccati:2013:fine-gaussian-fluctuations%
			,lachieze-rey-peccati:2013:fine-gaussian-fluctuations1%
		}
		which also offer contraction bounds for normal
		approximation on the Poisson space.
	\end{example}
	
	\section{Proof of main results}
	\label{sec:proof-main-results}
	We begin with the proof of Theorem \ref{theorem_fourmomentHilbert}
	which uses the method of exchangeable pairs developed in Section \ref{Section_prelim}.
	\subsection{Proof of Theorem \ref{theorem_fourmomentHilbert}}
	\label{Section_prooffourmomentHilbert}
	Let $G$ be a Gaussian random variable on $K$ with the same
	covariance operator as $X$, i.e., $G$ has covariance operator
	$S$. Similarly to \cite[Corrolary 3.3]{bourguin-campese:2020:approximation-hilbert-valued-gaussians}, it holds that
	\begin{align*}
		d_3\brac{G,Z}\leq \frac{1}{2}\norm{S-S'}_{\operatorname{HS}}.
	\end{align*}
	Therefore, it suffices to derive the desired moment bound for
	$d_3\brac{X,G}$ which yields the first item in Theorem
	\ref{theorem_fourmomentHilbert} as
	\begin{align*}
		d_3\brac{X,Z}\leq d_3\brac{X,G}+d_3\brac{G,Z}.
	\end{align*}
	In Subsection
	\ref{constructionofthepair}, we constructed an exchangeable pair
	of the form $(F_q,F_q^t)$ based on an element of a fixed $K$-valued
	chaos $F_q$, where $q$ denotes the order of the Poisson chaos. Recall that $X$ has the
	chaos decomposition \eqref{chaosdecompofX}. It follows that, for any
	$t \geq 0$,  if we
	define $X^t$ as
	\begin{equation*}
		X^t=\sum_{q=1}^{N}F_q^t, 
	\end{equation*}
	then the pair $(X,X^t)$ is also exchangeable. Since $\inner{x-y,Dg(x)+Dg(y)}_K$ is an anti-symmetric expression, the exchangeability implies
		\begin{align*}
			\lim_{t\to 0}\frac{1}{2t}\E{\inner{-L^{-1}(X^t-X),Dg(X^t)+Dg(X)}_K}=0.
	\end{align*}
	Furthermore, applying Taylor's theorem yields
	\begin{align*}
		0&=\lim_{t\to 0}\frac{1}{2t}\E{\inner{-L^{-1}(X^t-X),Dg(X^t)+Dg(X)}_K}\\
		&= \lim_{t\to 0}\E{\frac{1}{2t}\inner{-L^{-1}(X^t-X),Dg(X^t)-Dg(X)}_K+\frac{1}{t}\inner{-L^{-1}(X^t-X),Dg(X)}_K}\\
		&=\lim_{t\to 0}\E{\frac{1}{2t}\inner{-L^{-1}(X^t-X),D^2g(X)(X^t-X)+r}_K+\frac{1}{t}\inner{-L^{-1}(X^t-X),Dg(X)}_K}.
	\end{align*}
	Here, $r$ denotes the remainder term for which
		$\norm{r}_K\leq \frac{1}{2}\norm{D^3g(\xi) (X^t-X)^2}_K$,
		and $\xi$ is in the open ball centered at $X$ with radius $\norm{X_t-X}_K$.
	
	Now let $R(t)=\E{\frac{1}{2t}\inner{-L^{-1}(X^t-X),r}_K}$. Note that
	$\E{\Delta_G g(X)}=\sum_{1\leq q\leq
		N}\E{\Tr_K\brac{D^2g(X)S_q}}$. Combined with part (a) and
	(b) of Lemma \ref{lemmaexchangeablepair} and keeping in mind $F_q=\sum_{i\in \N}F_{q,i}k_i$, this leads to
	\begin{align*}
		0       &=\sum_{1\leq q\leq
			N}\E{\Tr_K\brac{D^2g(X)\Gamma
				\brac{F_q,-L^{-1}F_q}}}\\ & \quad +\sum_{1\leq p\neq q\leq
			N}\sum_{i,j\in\N}\E{\inner{k_i,D^2g(X)\widetilde{\Gamma}\brac{-\widetilde{L}^{-1}F_{p,i},F_{q,j}}k_j
			}_K} -\E{\inner{X,Dg(X)}_K}+\lim_{t\to 0}R(t)\\       
		&=\E{\Delta_G g(X)}-\E{\inner{X,Dg(X)}_K}
		+\sum_{1\leq q\leq
			N}\E{\Tr_K\brac{D^2g(X)\brac{\Gamma
					\brac{F_q,-L^{-1}F_q}-S_q}}}\\ & \quad +\sum_{1\leq p\neq q\leq N}\sum_{i,j\in\N}\E{\inner{k_i,D^2g(X)\widetilde{\Gamma}\brac{-\widetilde{L}^{-1}F_{p,i},F_{q,j}}k_j }_K}+\lim_{t\to 0}R(t).
	\end{align*}
	The above equation and the Stein equation introduced in Section \ref{Section_prelim} imply
	\begin{align}
		\label{estimate_stein}
		d_3(X,G)&=  \sup_{h\in C^3_b(K)}\abs{\Delta_G g(X)-\inner{X,Dg(X)}_K}\nonumber\\
		& \leq \sup_{h\in C^3_b(K)}\left\{    \sum_{1\leq q\leq
			N}\abs{\E{\Tr_K\brac{D^2g(X)\brac{\Gamma
						\brac{F_q,-L^{-1}F_q}-S_q}}}}\right.\nonumber\\ &
		\quad \left. +\sum_{1\leq p\neq q\leq N}\abs{\sum_{i,j\in\N}\E{\inner{k_i,D^2g(X)\widetilde{\Gamma}\brac{-\widetilde{L}^{-1}F_{p,i},F_{q,j}}k_j }_K}}+\abs{\lim_{t\to 0}R(t)} \right\}.
	\end{align} 
	For the first term on the right side of
	\eqref{estimate_stein}, it holds that
	\begin{align*}
		&\sum_{1\leq q\leq
			N}\abs{\E{\Tr_K\brac{D^2g(X)\brac{\Gamma
						\brac{F_q,-L^{-1}F_q}-S_q}}}} \\
		& \qquad\qquad\qquad\qquad\qquad \leq \sum_{1\leq q\leq N}\norm{D^2g(X)}_{L^2(\Omega;\operatorname{HS}(K))} \norm{\frac{1}{q}\Gamma(F_q,F_q)-S_q}_{L^2(\Omega;\operatorname{HS}(K))}\nonumber\\
		& \qquad\qquad\qquad\qquad\qquad \leq \sum_{1\leq q\leq N}\frac{1}{2q}\sqrt{\sum_{i,j\in \N} \V{\Gamma\brac{F_{q,i},F_{q,j}}}}\nonumber\\
		& \qquad\qquad\qquad\qquad\qquad\leq \sum_{1\leq q\leq N}\frac{2q-1}{4q}\sqrt{\sum_{i,j\in \N}\E{F_{q,i}^2F_{q,j}^2}-\E{F_{q,i}^2}\E{F_{q,j}^2}-2\E{F_{q,i}F_{q,j}}^2}\nonumber\\ 
		& \qquad\qquad\qquad\qquad\qquad = \sum_{1\leq q\leq N}\frac{2q-1}{4q}\sqrt{\E{\norm{F_q}^4_K}-\E{\norm{F_q}^2_K}^2-2\norm{S_q}^2_{\operatorname{HS}}}.
	\end{align*}
	In particular, we have used the
	fact that $\norm{D^2g(x)}_{K^{\otimes
			2}}=\norm{D^2g(x)}_{\operatorname{HS}(K)}$ and
	\cite[Lemma 2.4]{bourguin-campese:2020:approximation-hilbert-valued-gaussians} to get the third line above. The fourth
	line is a consequence of Lemma \ref{lemma_DVZlemma2.2}. Finally, the
	identity $\inner{Sf,g}_K=\E{\inner{X,f}_K\inner{X,g}_K}$
	allows us to get the term $\norm{S_q}_{\operatorname{HS}}$ in the last
	line.

	Now we study the second term on the right side of
	\eqref{estimate_stein}. Application of \cite[Lemma
	2.4]{bourguin-campese:2020:approximation-hilbert-valued-gaussians} and Lemma \ref{lemma_DVZlemma2.2} gives
	\begin{align*}
		&\sum_{1\leq p\neq q\leq N}\abs{\sum_{i,j\in\N}\E{\inner{k_i,D^2g(X)\widetilde{\Gamma}\brac{-\widetilde{L}^{-1}F_{p,i},F_{q,j}}k_j }_K}}\\
		&\qquad\qquad\qquad\leq \sum_{1\leq p\neq q\leq N}\E{\sqrt{\sum_{i,j\in\N}\inner{k_i,D^2g(X)k_j }_K } \sqrt{\sum_{i,j\in\N}\widetilde{\Gamma}\brac{-\widetilde{L}^{-1}F_{p,i},F_{q,j}}^2}}\\
		&\qquad\qquad\qquad\leq \sum_{1\leq p\neq q\leq N}\sqrt{\sum_{i,j\in\N}\E{\inner{k_i,D^2g(X)k_j }^2_K }} \sqrt{\sum_{i,j\in\N}\E{\widetilde{\Gamma}\brac{-\widetilde{L}^{-1}F_{p,i},F_{q,j}}^2}}  \\
		&\qquad\qquad\qquad \leq  \sum_{1\leq p\neq q\leq N}\frac{p+q-1}{2p}\norm{D^2g(X)}_{L^2(\Omega;\operatorname{HS}(K))}\sqrt{\sum_{i,j\in \N}\E{F_{p,i}^2F_{q,j}^2}-\E{F_{p,i}^2}\E{F_{q,j}^2}}\\
		&\qquad\qquad\qquad\leq \sum_{1\leq p\neq q\leq N}\frac{p+q-1}{4p}\sqrt{\E{\norm{F_p}^2_K\norm{F_q}^2_K}-\E{\norm{F_p}^2_K}\E{\norm{F_q}^2_K}}.
	\end{align*}
	As the last step, we apply Lemma \ref{lemma_remaindertermbound} to the remainder term in \eqref{estimate_stein}.
		\begin{align*}
			\abs{\lim_{t\to 0}R(t)}&\leq \lim_{t\to 0}\frac{1}{4t}\norm{D^3g}_{\operatorname{op}}\E{\norm{-L^{-1}\brac{X_t-X}}_K\norm{X^t-X}^2_K}\\
			&\leq \frac{N}{2}\sqrt{\max_{1\leq p\leq N}\E{\norm{F_p}^2_K}}\sum_{1\leq q\leq N} \sqrt{4q-3}\sqrt{{\norm{F_q}_K^4-\E{\norm{F_q}^2_K}^2-2\norm{S_q}^2_{\operatorname{HS}}}}.
	\end{align*}
	We can hence deduce from \eqref{estimate_stein} the inequality
	\begin{align}
		\label{estimate_first_G}
		d_3(X,G)&\leq\sum_{1\leq q\leq
			N}\frac{2q-1}{4q}\sqrt{\E{\norm{F_q}^4_K}-\E{\norm{F_q}^2_K}^2-2\norm{S_q}^2_{\operatorname{HS}}}\nonumber\\
		&+\sum_{1\leq p\neq q\leq
			N}\frac{p+q-1}{4p}\sqrt{\E{\norm{F_p}^2_K\norm{F_q}^2_K}-\E{\norm{F_p}^2_K}\E{\norm{F_q}^2_K}}\nonumber\\
		&+\frac{N}{2}\sqrt{\max_{1\leq p\leq N}\E{\norm{F_p}^2_K}}\sum_{1\leq q\leq N} \sqrt{4q-3}\sqrt{{\norm{F_q}_K^4-\E{\norm{F_q}^2_K}^2-2\norm{S_q}^2_{\operatorname{HS}}}}.
	\end{align}
	Now, to get the second estimate in Theorem \ref{theorem_fourmomentHilbert}, observe that
	\begin{align*}
		\E{\norm{X}^4_K}-(\E{\norm{X}^2_K})^2-2\norm{S}^2_{\operatorname{HS}}
		=&\sum_{1\leq q\leq N} \E{\norm{F_q}^4_K}-\E{\norm{F_q}^2_K}^2-2\norm{S_q}^2_{\operatorname{HS}}\\
		&+\sum_{1\leq p\neq q\leq N}\E{\norm{F_p}^2_K\norm{F_q}^2_K}-\E{\norm{F_p}^2_K}\E{\norm{F_q}^2_K},
	\end{align*}
	This combined with Lemma
	\ref{lemmapositive4thmoment}, the bound at
	\eqref{estimate_first_G} and 
\begin{equation*}
        \begin{cases}
       		\displaystyle \frac{2q-1}{4q}\vee \frac{p+q-1}{4p}\leq \frac{2N-1}{4}&\text{ for $1\leq p,q\leq N$},\\
		\displaystyle \sum_{1\leq q,p\leq N}\sqrt{y_{q,p}}\leq \sqrt{N^2 \sum_{1\leq p,q\leq N}y_{q,p}}&\text{ for $y_{q,p}\geq 0$},\\
		\displaystyle 4q-3\leq 4N-3&\text{ for $1\leq q\leq N$},   
        \end{cases}
\end{equation*}
	yields
	\begin{align*}
		d_3(X,G)&\leq \brac{\frac{N(2N-1)}{4}+\frac{N}{2}\sqrt{(4N-3)\max_{1\leq p\leq N}\E{\norm{F_p}^2_K}}}
		\\ &\hspace{10em}\sqrt{\E{\norm{X}^4_K}-\E{\norm{X}^2_K}^2-2\norm{S}^2_{\operatorname{HS}}}. 
	\end{align*}
	\qed

	We now turn to the proof of Theorem \ref{theorem_contractionestimate},
	which makes use of the second estimate in Theorem \ref{theorem_fourmomentHilbert}.
	\subsection{Proof of Theorem
		\ref{theorem_contractionestimate}}    
	The strategy here consists of making use of the product
	formula \eqref{productformulaPoisson} for Poisson multiple integrals in order to represent
	the quantity
	$\E{\norm{X}^4_K}-\E{\norm{X}^2_K}^2-2\norm{S}^2_{\operatorname{HS}}$
	which appears in the second estimate of Theorem
	\ref{theorem_fourmomentHilbert} in term of contraction
	norms. We begin by noting that this quantity can be written as
	\begin{align*}
		\E{\norm{X}^4_K}-\E{\norm{X}^2_K}^2-2\norm{S}^2_{\operatorname{HS}}=&\sum_{\substack{i,j\in
				\N\\1\leq
				p,q\leq N}}\left(
		\E{F_{q,i}^2F_{p,j}^2}-\E{F_{q,i}^2}\E{F_{p,j}^2}-2\E{F_{q,i}F_{p,j}}^2
		\right)\\
		=& \sum_{\substack{i,j\in \N\\1\leq  p,q\leq N}}\left(
		\E{F_{q,i}^2F_{p,j}^2}-\E{F_{q,i}^2}\E{F_{p,j}^2}\right)\\
		&-2\sum_{\substack{i,j\in \N\\1\leq  q\leq N}}\E{F_{q,i}F_{p,j}}^2 .
	\end{align*}
	An application of the product formula
	\eqref{productformulaPoisson} for Poisson multiple integrals yields
	\begin{equation*}
		F_{q,i}F_{p,j}=\sum_{r=0}^{q\wedge p} r!{q \choose r}{p\choose r}\sum_{l=0}^r{r\choose l}I_{q+p-r-l}\brac{{f_{q,i}
				{\widetilde{\star}}^l_rf_{p,j}}}.
	\end{equation*}
	Now by the orthogonality of
	Poisson chaos of different orders, one has
	\begin{align}
		\label{fourthmoment}
		\E{F_{q,i}^2F_{p,j}^2}&=\sum_{r,s=0}^{q\wedge p} \sum_{\substack{0\leq l\leq r\\0\leq m\leq s\\r+l=s+m}}c_{p,q,l,m}(r,s) \inner{{f_{q,i}
				\widetilde{\star}^l_rf_{p,j}},{f_{q,i}
				\widetilde{\star}^m_sf_{p,j}}}_{\mathfrak{H}^{\otimes (q+p-r-l)}},
	\end{align}
	where the coefficient $c_{p,q,l,m}(r,s)$ is given by
	\begin{equation*}
		c_{p,q,l,m}(r,s) = r!s! \binom{q}{r}\binom{q}{s}\binom{p}{r}\binom{p}{s}\binom{r}{l}\binom{s}{m}(p+q-r-l)!.
	\end{equation*}
	Let us define the index set $I$ as
	\begin{align*}I=\big\{(r,s,l,m)\in \N^4\colon & 0\leq
		r,s\leq q\wedge p,\ 0\leq l\leq r,\ 0\leq m\leq s,\\
		&
		r+l=s+m,\ (r,s,l,m)\notin \{(0,0,0,0),(q\wedge p,q\wedge p,q\wedge p,q\wedge p) \}\big\}.\end{align*}
	% and coefficients
	% \begin{align*}
		%       a_{p,q}(c)=p!q!{q \choose c}{p\choose c},b_{p,q}(r)=r!{q \choose r}{p\choose r}.
		% \end{align*} 
	Then, using Lemma \ref{lemma_contraction_00},  Equation \eqref{fourthmoment} can be rewritten as
	\begin{align*}
		\E{F_{q,i}^2F_{p,j}^2}=&q!p!\norm{f_{q,i}}^2_{\mathfrak{H}^{\otimes
				q}}\norm{f_{p,j}}^2_{\mathfrak{H}^{\otimes
				q}}+2q!^2\inner{f_{q,i},f_{q,j}}^2_{\mathfrak{H}^{\otimes
				q}}\\
		&+a_{p,q}\brac{p\wedge q}\norm{f_{q,i} \star^{q\wedge
				p}_{q\wedge p} f_{p,j}}^2_{\mathfrak{H}^{\otimes
				\abs{q-p}}}\mathds{1}_{\left\{ q\neq p \right\}}+\sum_{r=1}^{q\wedge p -1}b_{p,q}\brac{r} \norm{f_{q,i} \star^r_r f_{p,j}}^2_{\mathfrak{H}^{\otimes (q+p-2r)}}\\
		&+\sum_{(r,s,l,m)\in I}c_{p,q,l,m}(r,s)\inner{{f_{q,i}
				\widetilde{\star}^l_rf_{p,j}},{f_{q,i}
				\widetilde{\star}^m_sf_{p,j}}}_{\mathfrak{H}^{\otimes (q+p-r-l)}},
	\end{align*}
	where the combinatorial coefficients $a_{p,q}(r)$ and $b_{p,q}(r)$ are given by
	\begin{equation*}
		\begin{cases}
			\displaystyle a_{p,q}(r) = p!q! \binom{q}{r}\binom{p}{r} + r!^2 \binom{q}{r}^2
			\binom{p}{r}^2 \abs{p-q}! \\
			\displaystyle b_{p,q}(r) = p!q! \binom{q}{r}\binom{p}{r}
		\end{cases}.
	\end{equation*}
	Consequently, we hence obtain
	\begin{align*}
		\E{\norm{X}^4_K}-\E{\norm{X}^2_K}^2-2\norm{S}^2_{\operatorname{HS}}=&\sum_{\substack{i,j\in \N\\1\leq  p,q\leq N}}\left(  \E{F_{q,i}^2F_{p,j}^2}-\E{F_{q,i}^2}\E{F_{p,j}^2}-2\E{F_{q,i}F_{p,j}}^2\right)\\
		=&\sum_{\substack{i,j\in \N\\1\leq  p\neq q\leq N}}a_{p,q}\brac{p\wedge q}\norm{f_{q,i} \star^{q\wedge
				p}_{q\wedge p} f_{p,j}}^2_{\mathfrak{H}^{\otimes
				\abs{q-p}}}\\
		&+\sum_{\substack{i,j\in \N\\1\leq  p,q\leq N}}\sum_{r=1}^{q\wedge p -1}b_{p,q}\brac{r} \norm{f_{q,i} \star^r_r f_{p,j}}^2_{\mathfrak{H}^{\otimes (q+p-2r)}}\\
		&+\sum_{\substack{i,j\in \N\\1\leq  p,q\leq N\\(r,s,l,m)\in I}}c_{p,q,l,m}(r,s)\inner{{f_{q,i}
				\widetilde{\star}^l_rf_{p,j}},{f_{q,i}
				\widetilde{\star}^m_sf_{p,j}}}_{\mathfrak{H}^{\otimes (q+p-r-l)}}.
	\end{align*}
	Since we have
	\begin{align*}  
		\norm{{f_{q}{\star}^l_rf_{p}}}^2_{\mathfrak{H}^{\otimes (q+p-r-l)}\otimes K^{\otimes 2}}=\sum_{i,j\in\N}\norm{{\inner{f_{q},k_i}_K{\star}^l_r \inner{f_{p},k_j}_K}}^2_{\mathfrak{H}^{\otimes (q+p-r-l)}}=\sum_{i,j\in\N}\norm{{f_{q,i}{\star}^l_rf_{p,j}}}^2_{\mathfrak{H}^{\otimes (q+p-r-l)}},
	\end{align*}
	we can sum over $i,j \in \N$ and apply Holder's inequality to get
	\begin{align*}
		&\E{\norm{X}^4_K}-\E{\norm{X}^2_K}^2-2\norm{S}^2_{\operatorname{HS}} \leq \sum_{\substack{i,j\in \N\\1\leq  p\neq q\leq N}}a_{p,q}\brac{p\wedge q}\norm{f_{q,i} \star^{q\wedge
				p}_{q\wedge p} f_{p,j}}^2_{\mathfrak{H}^{\otimes
				\abs{q-p}}}\\
		&\qquad\qquad\qquad\qquad\qquad\qquad\qquad\quad\quad+\sum_{\substack{i,j\in \N\\1\leq  p,q\leq N}}\sum_{r=1}^{q\wedge p -1}b_{p,q}\brac{r} \norm{f_{q,i} \star^r_r f_{p,j}}^2_{\mathfrak{H}^{\otimes (q+p-2r)}}\\
		&\qquad\qquad\qquad\qquad\qquad\qquad+\sum_{\substack{i,j\in \N\\1\leq  p,q\leq N\\(r,s,l,m)\in I}}c_{p,q,l,m}(r,s)\norm{f_{q} \star^l_r f_{p}}_{\mathfrak{H}^{\otimes (q+p-r-l)}}\norm{f_{q} \star^m_s f_{p}}_{\mathfrak{H}^{\otimes (q+p-r-l)}},
	\end{align*}    
	which concludes the proof.      \qed
	
	\section{Applications}
	\label{sec:applications}
	\subsection{Brownian approximation of a Poisson process in Besov-Liouville spaces}
	\label{subsection_Besov}
	
	% TODO: Compare with Coutin-Decreusefond. In particular, point out that we obtain the same rate (which has also classically been obtained in older papers, I think for example by Barbour) but no longer are in $l^{2}$ but in the true Besov-Liouville space.
	
	\subsubsection{A brief overview of Besov-Liouville spaces}
	For an extensive account on the current topic, we invite readers to view \cite{samko-kilbas-marichev:1993:fractional-integrals-derivatives}. For $f\in L^p([0,1],ds)$ and $\beta>0$, we define the left and right fractional integrals respectively as
	\begin{align*}
		\brac{I^\beta_{0^{+}}f}(s)=\frac{1}{\Gamma(\beta)}\int_0^s (s-r)^{\beta-1}f(r)dr
	\end{align*}
	and
	\begin{align*}
		\brac{I^\beta_{1^{-}}f}(s)=\frac{1}{\Gamma(\beta)}\int_s^1 (r-s)^{\beta-1}f(r)dr.
	\end{align*}
	This allows us to define the Besov-Liouville spaces
	\begin{align*}
		\mathcal{I}^+_{\beta,p}=\left\{ I^\beta_{0^{+}}\widehat{f},\ \widehat{f}\in L^p([0,1])\right\},
	\end{align*}
	which are Banach spaces when equipped with the norm $\norm{f}_{\mathcal{I}^+_{\beta,p}}=\norm{\widehat{f}}_{L^p([0,1])}$.
	The Besov-Liouville spaces $\mathcal{I}^-_{\beta,p}$ are defined accordingly with the right fractional integrals. When $\beta p<1$,  
	the spaces $\mathcal{I}^+_{\beta,p}$ and
	$\mathcal{I}^-_{\beta,p}$ are canonically isomorphic and
	therefore will both be denoted by
	$\mathcal{I}_{\beta,p}$.
	
	\begin{remark}
		As pointed out in
		\cite{coutin-decreusefond:2013:steins-method-brownian}, $\mathcal{I}_{\beta,2}$ for $\beta<1/2$ is an
		appropriate class of Besov-Liouville spaces for the functional
		approximation of a Poisson process by a Brownian motion since
		they are Hilbert spaces containing both the sample paths of the
		Poisson process and the Brownian motion. 
	\end{remark}

	Similarly to the left and right fractional integrals, one can define left and right fractional derivatives as
	\begin{align*}
		\brac{D^\beta_{0^{+}}f}(s)=\frac{1}{\Gamma(1-\beta)}\frac{d}{ds}\int_0^s (s-r)^{-\beta}f(r)dr\\
		\brac{D^\beta_{1^{-}}f}(s)=\frac{1}{\Gamma(1-\beta)}\frac{d}{ds}\int_s^1 (r-s)^{-\beta}f(r)dr
	\end{align*}
	As the name suggests, $D^\beta_{0^{+}}$ is  the inverse of $I^\beta_{0^{+}}$ (see \cite[Theorem 2.4]{samko-kilbas-marichev:1993:fractional-integrals-derivatives}). Two examples for the action of this operator that will be useful later are
	\begin{align}
		\label{example_fracder}
		\brac{D^\beta_{0^{+}}\operatorname{Id}}
		(r)=\frac{r^{-\beta+1}}{(-\beta+1)\Gamma(-\beta+1)}\quad
		\mbox{and}\quad \brac{D^\beta_{0^{+}}1_{[a,\infty)}} (r)=\frac{\brac{r-a}^{-\beta}_{+}}{\Gamma(-\beta+1)},
	\end{align}
	where $\operatorname{Id}$ denotes the identity function. Let
	us also mention a few important facts about fractional integrals
	and derivatives. Given $0<\beta<1$ and $1<p<1/\beta$,
	$I^\beta_{0^{+}}$ is a bounded operator from $L^p([0,1])$ to
	$L^q([0,1])$ with $q=p(1-\beta p)^{-1}$. Moreover, for $\beta>0$
	and $p\geq 1$, $I^\beta_{0^{+}}$ is bounded from  $L^p([0,1])$
	into itself (see for instance \cite[Equation (2.72)]{samko-kilbas-marichev:1993:fractional-integrals-derivatives}). Next, fractional
	derivatives are the inverses of fractional integrals, in the
	sense that 
	\begin{align*}\brac{D^\beta_{0^{+}} I^\beta_{0^{+}}f} (s)=f(s)
	\end{align*}
	for $f\in L^1([0,1])$. Furthermore, fractional integrals enjoy the
	semigroup property (see \cite[Theorem 2.5]{samko-kilbas-marichev:1993:fractional-integrals-derivatives}), that is
	\begin{align*}
		\brac{I^\alpha_{0^{+}}I^\beta_{0^{+}}f}(s)=\brac{I^{\alpha+\beta}_{0^{+}}f}(s)
	\end{align*}
	as long as $\beta>0$, $\alpha+\beta>0$ and $f\in L^1([0,1])$.   
	\subsubsection{A functional central limit theorem}      
	\label{subsubsection_theorem_Besov}
	We consider a Poisson process $N_\lambda (t)$ with intensity
	$\lambda$. It is well known (see for instance \cite[Example
	9.1.3]{nualart-nualart:2018:introduction-malliavin-calculus}) that it can be represented as
	\begin{align}
		\label{def_poiprocess_Besov}
		N_\lambda (t)&=\sum_{n\in \N} 1_{[T_n,\infty)}(t),
	\end{align}
	where $T_n=\sum_{i=1}^n \alpha_i$ and $\left\{  \alpha_i
	\colon i \in \N\right\}$ are independent exponentially distributed random
	variables with parameter $\lambda$, i.e.,   $\alpha_i \sim
	\operatorname{Exp}(\lambda)$ for all $i \in \N$. This implies
	that $T_n$ is Gamma distributed with shape $n$ and rate
	$\lambda$, i.e., $T_n\sim \operatorname{Gamma}(n,\lambda)$. As
	pointed out in \cite{coutin-decreusefond:2013:steins-method-brownian}, $N_\lambda (t)$ maps into
	$\mathcal{I}_{\beta,2}$ for $\beta<1/2$. 
	
	For any $t \in
	[0,1]$, define 
	\begin{align*}
		X_\lambda (t)&=\frac{N_\lambda (t)-\lambda t}{\sqrt{\lambda}}
	\end{align*}
	and let $Z$ be a Brownian motion on $\mathcal{I}_{\beta,2}$,
	that is a $\mathcal{I}_{\beta,2}$-valued Gaussian random variable with covariance operator 
	\begin{align}
		\label{covariance_BM_Besov} 
		S'=I^\beta_{0^{+}} I^{1-\beta}_{0^{+}} I^{1-\beta}_{1^{-}} D^\beta_{0^{+}},
	\end{align}
	where the expression of the covariance operator was derived in
	\cite{coutin-decreusefond:2013:steins-method-brownian}. We are now ready to state the main result of this
	application, namely the Brownian approximation of a Poisson
	process in $\mathcal{I}_{\beta,2}$.
	\begin{theorem}
		\label{theorem_Besov}
		On a Besov-Liouville space $\mathcal{I}_{\beta,2}$
		with $\beta<1/2$, the distributions of $X_\lambda$ and
		$Z$ are asymptotically close as
		$\lambda\to\infty$. Their closeness can be quantified
		by
		\begin{align*}
			d_3(X_\lambda,Z)\lesssim \frac{1}{\sqrt{\lambda}}.
		\end{align*}
	\end{theorem}

	\begin{proof}
		$X_\lambda (t)$ can be represented as
		a Poisson multiple integral of order one. Let
		$\mathfrak{H}=L^2(\R^+,\lambda dx)$ be the underlying
		Hilbert space to the compensated Poisson process
		$N_\lambda(t)-\lambda t$. Furthermore, let
		$f(t)=\frac{1}{\sqrt{\lambda}}1_{[0,t]}\in
		\mathfrak{H}$. We can hence write
		\begin{align*}
			X_\lambda (t)=I_1(f(t)).
		\end{align*}    
		Theorem \ref{theorem_contractionestimate} then
		provides us with the estimate
		\begin{align}
			\label{bound_stein_Besov}
			d_3(X_\lambda,Z)\lesssim \norm{f \star^0_1 f}^2_{\mathfrak{H}\otimes K^{\otimes 2}}+\norm{S_\lambda-S'}_{\operatorname{HS}(K)},
		\end{align}
		where $S_\lambda$ denotes the covariance operator of
		$X_{\lambda}$ and where $K=\mathcal{I}_{\beta,2}$. We begin by computing the contraction
		norm appearing above. We have
		\begin{align*}
			(f \star^0_1 f)(x)=\frac{1}{\lambda}1_{[0,t]}(x)1_{[0,s]}(x)=1_{[x,\infty)}(t)1_{[x,\infty)}(s),\end{align*}
		so that
		\begin{align*}
			\norm{f \star^0_1 f}^2_{\mathfrak{H}\otimes K^{\otimes 2}}&=\frac{1}{\lambda^2}\int_0^1\int_0^1 \int_0^1 \brac{\brac{D^\beta_{0^{+}}1_{[x,\infty)}}(t) \brac{D^\beta_{0^{+}}1_{[x,\infty)}}(s)}^2 \lambda dxdsdt\\
			&=\frac{1}{\lambda \Gamma(-\beta+1)^4}\int_0^1\int_0^1 \brac{t-x}^{-2\beta}_{+}\brac{s-x}^{-2\beta}_{+}dsdt\lesssim \frac{1}{\lambda},
		\end{align*}
		where the last inequality simply comes from the fact that  $\int_0^1\int_0^1 \brac{t-x}^{-2\beta}_{+}\brac{s-x}^{-2\beta}_{+}dsdt$ is finite.
		
		Regarding the remaining term, namely
		$\norm{S_\lambda-S'}_{\operatorname{HS}(K)}$, we apply Lemma \ref{lemma_covariance}  
		and Lemma \ref{lemma_cov_kernel}. This yields 
		\begin{align*}
				\norm{S_\lambda-S'}_{\operatorname{HS}(K)}^2&=\norm{\E{\brac{D^\beta_{0^{+}}X_\lambda}(r)
                                                                              \brac{D^\beta_{0^{+}}X_\lambda}(s)}-\E{\brac{D^\beta_{0^{+}}Z}(r)
                                                                              \brac{D^\beta_{0^{+}}Z}(s)}}_{L^2([0,1]^{\otimes
                                                                              2})}^2\\
                  &=0,
		\end{align*}   
		which concludes the proof.              
	\end{proof}

	\subsection{Edge counting in random graphs}
	In \cite{lachieze-rey-peccati:2013:fine-gaussian-fluctuations}, the authors studied Gaussian fluctuations of
	real-valued $U$-statistics related to graphs generated by
	Poisson point processes. We will apply Theorem
	\ref{theorem_contractionestimate} to obtain a functional
	version of their results in all three regimes
	mentioned in \cite[Example 4.13]{lachieze-rey-peccati:2013:fine-gaussian-fluctuations}. Recall from Subsection
	\ref{constructionofthepair} the definition of a proper Poisson point process
	\begin{align*}
		\eta_\lambda=\sum_{i=1}^{\operatorname{Po}(\lambda)}\delta_{Y_i},
	\end{align*}
	where $\operatorname{Po}(\lambda)$ is a Poisson distribution
	on $\R$, while $\{Y_i \}_{i\in\N}$ is an i.i.d. sequence of $\R^d$-valued
	random variables distributed as $\ell$ and independent from
	$\operatorname{Po}(\lambda)$. For simplicity and illustration
	purposes, let us assume $\ell$ is the Lebesgue measure on
	$\R^d$. The control measure of  $\eta_\lambda$ is therefore
	\begin{align*}\mu_\lambda(\cdot)=\lambda \ell(\cdot).\end{align*}
	
	Let $G$ be a graph
	generated by $\eta_\lambda$, so that $G$ has the vertex set
	$\{Y_1,\ldots,Y_{\operatorname{Po}(\lambda)}\}$. In addition,
	let $W\subseteq \R^{d}$ be symmetric and bounded, i.e. $\ell(W)<\infty$. $W$ will serve as
	our original window in which we monitor the edges of $G$, and
	let $H_{\lambda}\subseteq \R^{2d}$ be a symmetric set which
	will serve as our original edge set. For $0\leq t\leq 1$,
	define        
	\begin{equation*}
		\begin{cases}
			\displaystyle W_t=t^{\frac{1}{2d}}W\\
			\displaystyle H_{\lambda,t}=t^{\frac{1}{2d}}H_{\lambda}\\
			\displaystyle \widehat{W}_t=\{x-y:x,y\in W_t \}\\
			\displaystyle \overline{H}_{\lambda,t}=\{x-y:x,y\in H_{\lambda,t}\}
		\end{cases}.
	\end{equation*}
	We will assume that any edge, written in pairs $(x,y)$, belongs to
	$H_{\lambda,t}$ if and only if $x-y\in \overline{H}_{\lambda,t}$. For example, this property holds for a disk graph with base edge set $\overline{H}_{\lambda}=B\brac{0,r_\lambda}$, an open ball of radius $r_\lambda$ at the origin. We note that compared to the setup in \cite{lachieze-rey-peccati:2013:fine-gaussian-fluctuations}, our window and edge set are not static but evolve with time.
	
	We are interested in a Poissonized $U$-statistics of the form
	\begin{align*}
		F_\lambda (t)=\sum_{\substack{(x,y)\in \eta^2_\lambda\\x\neq y}}1_{H_{\lambda,t} \cap W_{t}^2} (x,y)=\sum_{1=i_1< i_2}^{\operatorname{Po}(\lambda)}1_{H_{\lambda,t} \cap W_{t}^2} (Y_{i_1},Y_{i_2})
	\end{align*}
	which counts edges of $G$ that belong to the set
	$\overline{H}_{\lambda,t}$ and lie inside the window $W_t$ at
	time $t$. It is clear from the hypothesis that $\{{F}_\lambda
	(t)\}_{t\in [0,1]}$ as a process belongs to
	$K=L^2\brac{[0,1]}$. As proved in \cite{reitzner-schulte:2013:central-limit-theorems}, our $U$-statistic
	has a finite chaos expansion given by
	\begin{equation*}
		F_\lambda (t)=\E{F_\lambda (t)} + I_1\brac{f_{1}(t)}+I_2\brac{f_{2}(t)},
	\end{equation*}
	where the (functional) kernels $f_1(t)$ and $f_2(t)$ are given by
	\begin{equation*}
		\begin{cases}
			\displaystyle f_{1}(t)=2\int_{\R^d} 1_{H_{\lambda,t} \cap W_{t}^{ 2}} (x,y) \lambda dy\\
			\displaystyle f_{2}(t)=1_{H_{\lambda,t} \cap W_{t}^{ 2}} (x,y)
		\end{cases}.
	\end{equation*}
	Let $\bar{F}_\lambda (t)$ denote the centered and
	normalized version of $F_\lambda (t)$ given by
	\begin{align*}
		\bar{F}_\lambda (t)= \frac{F_\lambda
			(t)-\E{F_\lambda (t)}}{\sigma}=I_1\left(g_1(t)\right)+I_2\left(g_2(t)\right),
	\end{align*}
	where $\sigma^2=\V{F_\lambda (1)}$,
	$g_1(t)=\frac{f_{1}(t)}{\sigma}$ and
	$g_2(t)=\frac{f_{2}(t)}{\sigma}$. For
	convenience, we will also write $\ell_t$ for $\ell\brac{W_t}$
	and $\psi_{\lambda,t}$ for $\ell\brac{\overline{H}_{\lambda,t}
		\cap \widehat{W}_t}$. Using the scaling properties of the Lebesgue
	measure, we can write
	\begin{align*}
		\ell_t=\sqrt{t}\ell_1\quad \mbox{and} \quad \psi_{\lambda,t}= \sqrt{t}\psi_{\lambda,1}.
	\end{align*}
	We can actually compute $\sigma^2$ explicitly, using the
	orthogonality of Wiener chaos of different orders and the isometry
	property of Poisson multiple integrals. This yields
	\begin{align*}
		\sigma^2=&\norm{f_{1}(1)}_{L^2\brac{\mu_\lambda}}^2+\norm{f_{2}(1)}_{L^2\brac{\mu_\lambda^2}}^2\\
		=& 4\lambda^3\int_{\R^d}\brac{\int_{\R^d}
			1_{W_1}(x)1_{\overline{H}_{\lambda,1} \cap
				\widehat{W}_1}(y-x)   d(y-x)}^2  dx +
		\int_{\R^{2d}}1_{H_{\lambda,1} \cap W_1^{ 2}}
		(x,y) \lambda^2 dx dy\\
		=& 4\ell_1\lambda^3
		\psi_{\lambda,1}^2 + \ell_1\lambda^2\psi_{\lambda,1}.
	\end{align*}
	% such that
	% \begin{align*}
		%       \norm{f_{1}(1)}_{L^2\brac{\mu_\lambda}}^2&=\int_{Z} \brac{2\int_Z 1_{H_{\lambda,t} \cap W_{1}^{ 2}} (x,y) \lambda dy}^2 \lambda dx\\
		%       &=4\lambda^3\int_{Z}\brac{\int_Z
			%    1_{W_1}(x)1_{\overline{H}_{\lambda,1} \cap
				%    \widehat{W}_1}(y-x)   d(y-x)}^2  dx=4\ell_1\lambda^3
		%    \psi_{\lambda,1}^2 + \ell_1\lambda^2\psi_{\lambda,1}\\
		%       %
		%       \norm{f_{2}(1)}_{L^2\brac{\mu_\lambda^2}}^2&=\int_{Z^{ 2}}1_{H_{\lambda,1} \cap W_1^{ 2}} (x,y) \lambda^2 dx dy=\ell_1\lambda^2\psi_{\lambda,1}.
		% \end{align*}
	Based on the above expression for $\sigma^2$, we can consider three
	different regimes (similarly to what was done in \cite{lachieze-rey-peccati:2013:fine-gaussian-fluctuations}), namely
	\begin{enumerate}
		\item[-] Regime 1: $\lambda \psi_{\lambda,1}\to\infty$ as $\lambda\to\infty$;
		\item[-] Regime 2: $\lambda \psi_{\lambda,1}\to 1$ for $c>0$ as $\lambda\to\infty$;
		\item[-] Regime 3: $\lambda \psi_{\lambda,1}\to 0$ and $\lambda \sqrt{\psi_{\lambda,1}}\to \infty$ as $\lambda\to\infty$. 
	\end{enumerate}
	Within Regime 1, $\sigma^2$ is dominated by
	$\norm{f_{1}(1)}_{L^2(\mu_\lambda)}^2$ for large values of $\lambda$,
	which implies
	\begin{align*}
		\sigma^2\asymp 4\ell_1\lambda^3 \psi_{\lambda,1}^2,
	\end{align*}
	whereas in Regime 2, we get
	\begin{align*}
		\sigma^2\asymp 4\ell_1\lambda^3 \psi_{\lambda,1}^2\asymp \ell_1\lambda^2\psi_{\lambda,1},
	\end{align*}
	and finally in Regime 3, it holds that
	\begin{align*}
		\sigma^2\asymp\ell_1\lambda^2\psi_{\lambda,1}.
	\end{align*}
	We are now ready to present the
	application of our results to edge counting in random graphs.  
	\begin{theorem}
		\label{theorem_edgecounting}
		As $\lambda\to\infty$, $\bar{F}_\lambda (t)$
		converges in $K=L^2([0,1])$ to a $K$-valued Gaussian random
		variable $Z$ with covariance function
		$\phi(s,t)=\E{Z(s)Z(t)}$. More specifically, 
		\begin{enumerate}
			\item[-] In Regime 1, $\phi(t,s)=\sqrt{ts(t\wedge s)}$ and
			\begin{align*}
				d_3\brac{\bar{F}_\lambda,Z}\lesssim \lambda^{-\frac{1}{2}}+\frac{1}{\lambda\psi_{\lambda,1}};
			\end{align*}
			\item[-] In Regime 2, $\phi(t,s)=\frac{4\sqrt{ts(t\wedge s)}+t\wedge s}{5}$ and
			\begin{align*}
				d_3\brac{\bar{F}_\lambda,Z}\lesssim \lambda^{-\frac{1}{2}}+\abs{\lambda \psi_{\lambda,1}- 1};
			\end{align*}
			\item[-]  In Regime 3, $\phi(t,s)=t\wedge s$
			which implies that $Z$ is a Brownian motion,
			and
			\begin{align*}
				d_3\brac{\bar{F}_\lambda,Z}\lesssim \lambda^{-1}\psi_{\lambda,1}^{-1/2}+\lambda \psi_{\lambda,1}.
			\end{align*}
		\end{enumerate}
	\end{theorem}

	\begin{proof}
		In order to make use of Theorem
		\ref{theorem_contractionestimate}, we will need to evaluate
		contraction norms, but also the Hilbert-Schmidt norm of the
		difference between the covariance operators, i.e.,
		$\norm{S-S'}_{\operatorname{HS}}$. Let us start with this term
		before we turn to the contraction norms themselves. As before, $S_\lambda$ and $S'$ denotes the covariance operator of $\bar{F}_\lambda$ and $Z$ respectively. Based on \cite[Theorem 7.4.3]{hsing-eubank:2015:theoretical-foundations-functional} and how Hilbert-Schmidt norms are defined for integral operators, we can use
		\begin{align*}
			\norm{S_\lambda-S'}_{\operatorname{HS(K)}}=&\norm{\E{\bar{F}_\lambda(t)\bar{F}_\lambda(s)}-\E{Z(t)Z(s)}}_{L^2\brac{[0,1]^{\otimes 2}}}\\
			\leq& \norm{\E{\bar{F}_\lambda(t)\bar{F}_\lambda(s)}-\E{Z(t)Z(s)}}_{\infty}.
		\end{align*}
		Our task is hence to compute $\E{\bar{F}_\lambda(t)\bar{F}_\lambda(s)}$. We have
		\begin{equation*}
			\inner{{f}_{1}(t),{f}_{1}(s)}_{L^2(\mu_\lambda)}=4\lambda^3\psi_{\lambda,t}\psi_{\lambda,s}\ell_{t\wedge
				s}=\sqrt{ts(t\wedge
				s)}4\ell_{1}\lambda^3\psi_{\lambda,1}^2
		\end{equation*}
		and
		\begin{equation*}
			\inner{{f}_{2}(t),{f}_{2}(s)}_{L^2(\mu^2_\lambda)}=\lambda^2\psi_{\lambda,t\wedge s}\ell_{t\wedge s}=(t\wedge s) \ell_{1}\lambda^2\psi_{\lambda,1},
		\end{equation*}
		so that
		\begin{align*}
			\E{\bar{F}_\lambda(t)\bar{F}_\lambda(s)}=&\frac{\inner{{f}_{1}(t),{f}_{1}(s)}_{L^2(\mu_\lambda)}+\inner{{f}_{2}(t),{f}_{2}(s)}_{L^2(\mu^2_\lambda)}}{\sigma^2}\\
			=&\frac{\sqrt{ts(t\wedge s)}4\lambda\psi_{\lambda,1}+t\wedge s}{4\lambda\psi_{\lambda,1}+1}.
		\end{align*}
		At this step, we need to
		differentiate our analysis depending on what regime we
		are in. 
		\\~\\
		\textbf{Regime 1:} We assume here that $\lambda
		\psi_{\lambda,1}\to\infty$. The limiting
		covariance operator $S'$ then has covariance function
		$\phi(t,s)=\sqrt{ts(t\wedge s)}$. We can use the fact that 
		for $a\ll A,b\ll B$,
		\begin{align*}
			\abs{\frac{A+a}{B+b}-\frac{A}{B}}\lesssim \abs{\frac{a}{B}}+\abs{\frac{b}{B}}
		\end{align*}
		in order to deduce that
		\begin{align}
			\label{estimatecovarregime1}
			\norm{S_\lambda-S'}_{\operatorname{HS(K)}}\leq \sup_{1\leq s,t\leq M}\abs{\E{\bar{F}_\lambda(t)\bar{F}_\lambda(s)}-\phi(t,s)}\lesssim \frac{1}{\lambda\psi_{\lambda,1}}.
		\end{align}
		~\\             
		\textbf{Regime 2:} Here, $\lambda \psi_{\lambda,1}\to 1$, so that the
		limiting covariance function is given by $\phi(t,s)=\frac{4\sqrt{ts(t\wedge s)}+t\wedge s}{5}$. Moreover,
		\begin{align}
			\label{estimatecovarregime2}
			\norm{S_\lambda-S'}_{\operatorname{HS(K)}}\leq& \sup_{1\leq s,t\leq M}\abs{\E{\bar{F}_\lambda(t)\bar{F}_\lambda(s)}-\phi(t,s)}\nonumber\\
			=&\sup_{1\leq s,t\leq M}\abs{\frac{4\sqrt{ts(t\wedge s)}\lambda \psi_{\lambda,1}+t\wedge s}{4\lambda \psi_{\lambda,1}+1} -\frac{4\sqrt{ts(t\wedge s)}+t\wedge s}{5}}
			\lesssim \abs{\lambda \psi_{\lambda,1}- 1}.
		\end{align}
		~\\
		\textbf{Regime 3:} The fact that $\lambda
		\psi_{\lambda,1}\to 0$ implies in this case that the
		limiting covariance function is given by
		$\phi(t,s)=t\wedge s$, and we hence have
		\begin{align}
			\label{estimatecovarregime3}
			\norm{S_\lambda-S'}_{\operatorname{HS(K)}}\lesssim  \frac{\lambda^3\psi_{\lambda,1}^2}{\lambda^2\psi_{\lambda,1}}   \asymp\lambda \psi_{\lambda,1}.
		\end{align}
		~\\
		We now turn to the second part of the bound appearing
		in Theorem \ref{theorem_contractionestimate}, namely
		the contraction norms. We need to evaluate the
		norms of $g_1(t)\star^0_1g_1(t)$,
		$g_1(t)\star^0_1g_2(t)$,
		$g_1(t)\star^1_1g_2(t)$,
		$g_2(t)\star^0_1g_2(t)$,
		$g_2(t)\star^0_2g_2(t)$ and
		$g_2(t)\star^1_1g_2(t)$. The
		calculations we need to perform are very similar to the ones
		appearing in the proof of \cite[Theorem 4.7]{lachieze-rey-peccati:2013:fine-gaussian-fluctuations}, hence we
		will not provide full details and proceed straight to the
		result. Let us still include two examples of these
		calculations (the cases of the contractions
		$g_1(t)\star^0_1g_1(t)$ and $g_2(t)\star^1_1g_2(t)$) for the reader's convenience and for the sake of
		staying self-contained. Recall that $W_t$, $H_{\lambda,t}$ are
		symmetric sets, $W_t$ (respectively $H_{\lambda,t}$) is
		contained in $W_{t'}$ (respectively $H_{\lambda,t'}$) for $t\leq t'$, and that $\psi_{\lambda,t}=
		\sqrt{t}\psi_{\lambda,1}$, while $\ell_t=\sqrt{t}\ell_1<\infty$. We
		can then write
		\begin{align*}
			&\norm{f_{1}(t)\star^0_1
				f_{1}(s)}^2_{L^2(\mu_\lambda)\otimes K^{\otimes
					2}}\\& \qquad =\norm{\int_{\R^d} \brac{4\int_{Z^2}1_{H_{\lambda,t} \cap W_t^{ 2}} (x,y) 1_{H_{\lambda,s} \cap W_s^{ 2}} (x,u)\lambda dy\lambda du  }^2 \lambda dx}_{K^{\otimes 2}}\\
			& \qquad\leq 16\lambda^5\norm{\int_{\R^d} \brac{\int_{\R^{2d}}1_{H_{\lambda,s\vee t} \cap W_{s\vee t}^{ 2}} (x,y) 1_{H_{\lambda,s\vee t} \cap W_{s\vee t}^{ 2}} (x,u)dydu  }^2  dx}_{K^{\otimes 2}}\\
			& \qquad \asymp \lambda^5 \norm{\int_{\R^d}
				\brac{\int_{\R^{2d}}1_{W_{s\vee t}}(x)1_{\overline{H}_{\lambda,s\vee t}
						\cap \widehat{W}_{s\vee t}}(y-x)
					1_{\overline{H}_{\lambda,s\vee t} \cap
						\widehat{W}_{s\vee t}}(u-x)d(y-x)d(u-x) }^2 dx}_{K^{\otimes 2}}\\
			& \qquad \asymp  \lambda^5\norm{\ell_{s\vee t}\psi_{\lambda,s\vee t}^4}_{K^{\otimes 2}}\asymp \lambda^5\psi_{\lambda,1}^4
		\end{align*}
		and
		\begin{align*}
			&\norm{f_{2}(t)\star^1_1
				f_{2}(s)}^2_{L^2(\mu^2_\lambda)\otimes K^{\otimes 2}}\\
			&\qquad \leq \norm{\int_{\R^{2d}}\brac{\int_{\R^d} 1_{H_{\lambda,{s\vee t}} \cap W_{s\vee t}^{ 2}}(x,y)1_{H_{\lambda,{s\vee t}} \cap W_{s\vee t}^{ 2}} (x,u)\lambda dx }^2 \lambda^2dydu}_{K^{\otimes 2}}\\
			&\qquad = \lambda^4 \norm{\int_{\R^{4d}} 1_{H_{\lambda,{s\vee t}} \cap W_{s\vee t}^{ 2}}(x,y)1_{H_{\lambda,{s\vee t}} \cap W_{s\vee t}^{ 2}}(x,u)1_{H_{\lambda,{s\vee t}} \cap W_{s\vee t}^{ 2}}(v,y)1_{H_{\lambda,{s\vee t}} \cap W_{s\vee t}^{ 2}}(v,u)dxdydudv}_{K^{\otimes 2}}\\
			&\qquad \leq \lambda^4
			\bigg\lVert\int_{\R^{4d}}1_{W_{s\vee t}}(x)1_{\overline{H}_{\lambda,{s\vee t}}
				\cap
				\widehat{W}_{s\vee t}}(y-x)1_{\overline{H}_{\lambda,{s\vee t}}
				\cap \widehat{W}_{s\vee t}}(u-x) 1_{W_{s\vee t}}(x)1_{\overline{H}_{\lambda,{s\vee t}}
				\cap
				\widehat{W}_{s\vee t}}(y-v) \\&\qquad\qquad\qquad\qquad\qquad\qquad\qquad\qquad\qquad\qquad\qquad\qquad\qquad\quad dxd(y-x)d(u-v)d(v-y)\bigg\lVert_{K^{\otimes 2}}\\
			&\qquad\asymp \lambda^4\norm{\ell_{s\vee t}\psi_{\lambda,{s\vee t}}^3}_{K^{\otimes 2}}\asymp\lambda^4\psi_{\lambda,1}^3. 
		\end{align*}
		For the remaining contractions, performing similar calculations
		yields $\norm{f_{1}(t)\star^0_1 f_{2}(t)}^2_{L^2(\mu_\lambda^2)\otimes K^{\otimes 2}}\\\lesssim
		\lambda^4 \psi_{\lambda,1}^3$, $\norm{f_{2}(t)\star^0_1 f_{2 }(t)}^2_{L^2(\mu_\lambda^3)\otimes K^{\otimes 2}}\lesssim
		\lambda^3 \psi_{\lambda,1}^2$, $\norm{f_{2}(t)\star^0_2 f_{2}(t)}^2_{L^2(\mu_\lambda^2)\otimes K^{\otimes 2}}\lesssim
		\lambda^2\psi_{\lambda,1}$, and finally $\norm{f_{1}(t)\star^1_1 f_{2}(t)}^2_{L^2(\mu_\lambda)\otimes K^{\otimes 2}} \lesssim
		\lambda^5\psi_{\lambda,1}^4$. We split the remainder
		of the proof into three cases corresponding to the
		three possible regimes.
		\\~\\
		\textbf{Regime 1:} Here, $\lambda
		\psi_{\lambda,1}\to\infty$ as $\lambda\to\infty$, and
		since $\sigma^2\asymp \lambda^3 \psi_{\lambda,1}^2$,
		we have
		$\norm{g_1(t)\star^0_1g_1(t)}^2_{L^2(\mu_\lambda)\otimes K^{\otimes 2}}\\\lesssim
		\lambda^{-1}$, $\norm{g_1(t)\star^0_1
			g_2(t)}^2_{L^2(\mu_\lambda^2)\otimes K^{\otimes 2}}\lesssim
		\lambda^{-2}\psi_{\lambda,1}^{-1}$,
		$\norm{g_2(t)\star^0_1
			g_2(t)}^2_{L^2(\mu_\lambda^3)\otimes K^{\otimes 2}} \lesssim
		\lambda^{-3}\psi_{\lambda,1}^{-2}$,
		$\\ \norm{g_2(t)\star^0_2
			g_2(t)}^2_{L^2(\mu_\lambda^2)\otimes K^{\otimes 2}} \lesssim
		\lambda^{-4}\psi_{\lambda,1}^{-3}$,
		$\norm{g_2(t)\star^1_1
			g_2(t)}^2_{L^2(\mu_\lambda^2)\otimes K^{\otimes 2}}
		\lesssim \lambda^{-2}\psi_{\lambda,1}^{-1}$ and lastly
		$\\\norm{g_1(t)\star^1_1
			g_2(t)}^2_{L^2(\mu_\lambda)\otimes K^{\otimes 2}} \lesssim
		\lambda^{-1}$.  Note that all the above
		estimates are asymptotically bounded from
		above by $\lambda^{-1}$, and using \eqref{estimatecovarregime1}, the estimate in Theorem \ref{theorem_contractionestimate} yields
		\begin{align*}
			d_3\brac{\bar{F}_\lambda,Z}\lesssim \lambda^{-\frac{1}{2}}+\frac{1}{\lambda\psi_{\lambda,1}}.
		\end{align*}
		~\\
		\textbf{Regime 2:} As in this case, we have $\lambda
		\psi_{\lambda,1}\to 1$ as $\lambda\to\infty$, we get
		$\sigma^2\asymp\lambda^3
		\psi_{\lambda,1}^2\asymp\lambda^2\psi_{\lambda,1}$. Therefore, we can reuse the computations from Regime 1
		combined with \eqref{estimatecovarregime2} to get
		\begin{align*}
			d_3\brac{\bar{F}_\lambda,Z}\lesssim \lambda^{-\frac{1}{2}}+\abs{\lambda \psi_{\lambda,1}- 1}.
		\end{align*}
		~\\
		\textbf{Regime 3:} In this regime, $\lambda
		\psi_{\lambda,1}\to 0$ and $\lambda
		\sqrt{\psi_{\lambda,1}}\to \infty$ as
		$\lambda\to\infty$, so that
		$\sigma^2\asymp\lambda^2\psi_{\lambda,1}$. This
		allows us to deduce that
		$\norm{g_1(t)\star^0_1g_1(t)}^2_{L^2(\mu_\lambda)\otimes K^{\otimes 2}}
		\lesssim \lambda\psi^2_{\lambda,1}$,
		$\norm{g_1(t)\star^0_1
			g_2(t)}^2_{L^2(\mu_\lambda^2)\otimes K^{\otimes 2}} \lesssim
		\psi_{\lambda,1}$,
		$\norm{g_2(t)\star^0_1
			g_2(t)}^2_{L^2(\mu_\lambda^3)\otimes K^{\otimes 2}} \lesssim
		\lambda^{-1}$, $\norm{g_2(t)\star^0_2
			g_2(t)}^2_{L^2(\mu_\lambda^2)\otimes K^{\otimes 2}} \lesssim
		\lambda^{-2}\psi_{\lambda,1}^{-1}$,
		$\norm{g_2(t)\star^1_1
			g_2(t)}^2_{L^2(\mu_\lambda^2)\otimes K^{\otimes 2}}\\
		\lesssim \psi_{\lambda,1}$ and
		$\norm{g_1(t)\star^1_1
			g_2(t)}^2_{L^2(\mu_\lambda)\otimes K^{\otimes 2}} \lesssim
		\lambda\psi_{\lambda,1}^2$. Since
		$\lambda^{-2}\ll \psi_{\lambda,1}\ll
		\lambda^{-1}$, all terms listed are
		asymptotically bounded by
		$\lambda^{-2}\psi_{\lambda,1}^{-1}$. Combining
		this fact with \eqref{estimatecovarregime3} yields
		\begin{align*}
			d_3\brac{\bar{F}_\lambda,Z}\lesssim \lambda^{-1}\psi_{\lambda,1}^{-1/2}+\lambda \psi_{\lambda,1},
		\end{align*}
		which concludes the proof.
	\end{proof}

	\section*{Appendix}
	\label{sec:appendix}
	This section gathers ancillary lemmas used in the proofs of our main
	results as well as in the different applications presented in this
	paper. 
	\subsection{Lemmas related to the proofs of Theorems \ref{theorem_fourmomentHilbert} and \ref{theorem_contractionestimate}}
	Our first two lemma contain important results from \cite{dobler-vidotto-zheng:2018:fourth-moment-theorems} which we
	restate here for reader's convenience. 
	\begin{lemma}
		\label{lemmaexchangeablepairDVZ}
		Let $p,q \geq 1$ be integers, and let
		$F_q=I_q^{\eta}(f_q),G_p=I_p^{\eta}(g_p)$ and
		$F^t_q=I_q^{\eta^t}(f_q),G^t_p=I_p^{\eta^t}(g_p)$ be
		real-valued Poisson multiple integrals as constructed
		in Section \ref{Section_prelim}. Assume further that $\E{F_q^4},\E{G_p^4}<\infty$. Then, the following limits hold almost surely.
		\begin{enumerate}[label=(\alph*)]
			\item $\lim_{t\to 0}\frac{1}{t}\E{F^t_q-F_q|\eta}=-qF$
			\item $ \lim_{t\to 0}\frac{1}{t}\E{(F^t_q-F_q)(G^t_p-G_p)|\eta}=2\widetilde{\Gamma}(F_q,G_p)$
			\item $ \lim_{t\to 0}\frac{1}{t}\E{F^t_q(G^t_p-G_p)|\eta}=2\widetilde{\Gamma}(F_q,G_p)-pF_qG_p$
			\item $\lim_{t\to 0} \frac{1}{t}\E{(F^t_q-F_q)^4}=-4q\E{F_q^4}+12\E{F_q^2\widetilde{\Gamma}(F_q,F_q)}$.
		\end{enumerate}
	\end{lemma}
	\begin{proof}
		The proof of part $(a)$, $(b)$ and $(d)$ are in
		\cite[Proposition 3.2]{dobler-vidotto-zheng:2018:fourth-moment-theorems}. Part $(c)$ is a
		consequence of $(a)$ and $(b)$. 
	\end{proof}
	The following is an estimate from \cite[Lemma 2.2]{dobler-vidotto-zheng:2018:fourth-moment-theorems}, which improves upon a similar result in \cite[Lemma 3.1]{dobler-peccati:2018:fourth-moment-theorem}. 
		\begin{lemma}
			\label{lemma_DVZlemma2.2}
			Let $p,q \geq 1$ be integers, and let
			$F_q=I_q^{\eta}(f_q)$ and
			$F^t_q=I_q^{\eta^t}(f_q)$ be
			real-valued Poisson multiple integrals as constructed
			in Section \ref{Section_prelim}. {Assume further that $\E{F_q^4},\E{G_p^4}<\infty$.} Then, we have
			\begin{align*}
				\E{\widetilde{\Gamma}\brac{F_q,G_p}^2}\leq \brac{\frac{p+q-1}{2}}^2 \brac{ \E{F_{q}^2G_{p}^2}-\E{F_{q}^2}\E{G_{p}^2}-2\E{F_{q}G_{p}}^2}.
			\end{align*}
	\end{lemma}
	
	Our next lemma states a more general version of Lemma
	\ref{lemmaexchangeablepairDVZ}, part $(d)$.
	\begin{lemma}
		\label{lemma_limit_realvalued4thpower}
		Let $(X,X^t)$ be an exchangeable pair such that
		$X=\sum_{q\in\N}I_q^{\eta}(x_q)$ and
		$X^t=\sum_{q\in\N}I_q^{\eta^t}(x_q)$. Let the pairs
		$(Y,Y^t)$,$(U,U^t)$ and $(V,V^t)$ be defined in the
		same way. Assume further that $\E{X^4},\E{Y^4},\E{U^4}$ and $\E{V^4}$ are finite. Then, one has
		\begin{align*}
			\lim_{t\to\infty}\frac{1}{t}\E{(X^t-X)(Y_t-Y)(U^t-U)(V^t-V)}
			=&4\mathbb{E} \left[    \widetilde{\Gamma}(X,Y)UV+\widetilde{\Gamma}(X,V)YU\right.\\
			&\qquad\qquad\qquad\quad \left. +\widetilde{\Gamma}(X,U)YV+\widetilde{L}XYUV\right].
		\end{align*}
	\end{lemma}     
	\begin{proof}
		This limit is a consequence of exchangeability and
		Lemma \ref{lemmaexchangeablepairDVZ}. Indeed, denoting
		\begin{equation*}
			M_t = \frac{1}{t}\E{(X^t-X)(Y_t-Y)(U^t-U)(V^t-V)},
		\end{equation*}
		we can write
		\begin{align*}
			\lim_{t\to\infty}M_t=&2\lim_{t\to\infty}\frac{1}{t}\E{XYUV-X^tYUV-XY^tUV-XYU^tV-XYUV^t}\\&+2\lim_{t\to\infty}\frac{1}{t}\E{X^tY^tUV+X^tYU^tV+X^tYUV^t}\\
			=&2\lim_{t\to\infty}\frac{1}{t}\E{-\brac{X^t-X}YUV-X\brac{Y^t-Y}UV-XY\brac{U^t-U}V-XYU\brac{V^t-V}}\\
			&+2\lim_{t\to\infty}\frac{1}{t}\E{\brac{X^tY^t-XY}UV+\brac{X^tU^t-XU}YV+\brac{X^tV^t-XV}YU}\\
			=&2\mathbb{E}\big[-\widetilde{L}XYUV-X\widetilde{L}YUV-XY\widetilde{L}UV-XYU\widetilde{L}V+\widetilde{L}(XY)UV+\widetilde{L}(XU)YV\\
			&+\widetilde{L}(XV)YU\big]\\
			=&4\E{\widetilde{\Gamma}(X,Y)UV+\widetilde{\Gamma}(X,V)YU+\widetilde{\Gamma}(X,U)YV+\widetilde{L}XYUV}.
		\end{align*}
	\end{proof}
	
	\begin{lemma}
		\label{lemmapositive4thmoment}
		Let $X=\sum_{1\leq q\leq N}F_q$, where $F_q\in
		\mathcal{H}^q(K)$ with covariance operator
		$S_q$ and $\E{\norm{F_q}_K^4}<\infty$. Furthermore, letting $\left\{ k_i \right\}_{i
			\in \N}$ be an orthonormal basis of $K$, $F_q$ can
		be written as $\sum_{i\in \N}F_{q,i}k_i$, where
		$F_{q,i} = \left\langle F_q,k_i \right\rangle_K$. Then, it holds that
		\begin{align*}
			\E{F_{q,i}^2F_{p,j}^2}-\E{F_{q,i}^2}\E{F_{p,j}^2}-2\E{F_{q,i}F_{p,j}}^2\geq 0,
		\end{align*}
		which leads to 
		\begin{align*}
			\E{\norm{F_q}^4_K}-\E{\norm{F_q}^2_K}^2-2\norm{S_q}^2_{\operatorname{HS}}\geq 0
		\end{align*}
		and             
		\begin{align*}
			\E{\norm{F_q}^2_K\norm{F_p}^2_K}-\E{\norm{F_q}^2_K}\E{\norm{F_p}^2_K}\geq 0 \text{ when $q\neq p$.}
		\end{align*}
	\end{lemma}
	\begin{proof}
		Per Remark \ref{remark_whyfinitefourmoment}, $\E{\norm{F_q}^4_K}<\infty$ implies that $\forall i\in \N, \E{F_{q,i}^4}<\infty$. Then by \cite[Section 5]{dobler-vidotto-zheng:2018:fourth-moment-theorems}, we have that
		\begin{align*}
			\E{J_{p+q}(F_{q,i}F_{p,j})^2}\geq \E{F_{q,i}F_{p,j}}^2+\E{F_{q,i}^2}\E{F_{p,j}^2},
		\end{align*}
		which yields
		\begin{align*}
			\E{F_{q,i}^2F_{p,j}^2}-\E{F_{q,i}^2}\E{F_{p,j}^2}-2\E{F_{q,i}F_{p,j}}^2\geq&
			\E{F_{q,i}^2F_{p,j}^2}-\E{F_{q,i}F_{p,j}}^2-\E{J_{p+q}(F_{q,i}F_{p,j})^2}\\
			\geq& \E{\sum_{m=1}^{p+q-1}J_m(F_{q,i}F_{p,j})^2}\\\geq& 0. 
		\end{align*}
		The second and third inequalities in the statement of
		our lemma immediately follow, since     
		\begin{align*}
			\E{\norm{F_q}^4_K}-\E{\norm{F_q}^2_K}^2-2\norm{S_q}^2_{\operatorname{HS}}=&
			\sum_{i,j\in\N}\E{F_{q,i}^2F_{q,j}^2}-\E{F_{q,i}^2}\E{F_{q,j}^2}-2\E{F_{q,i}F_{q,j}}^2\\
			\geq& 0
		\end{align*}
		and when $q\neq p$,
		\begin{align*}
			\E{\norm{F_q}^2_K\norm{F_p}^2_K}-\E{\norm{F_q}^2_K}\E{\norm{F_p}^2_K}= \sum_{i,j\in\N}\E{F_{q,i}^2F_{p,j}^2}-\E{F_{q,i}^2}\E{F_{p,j}^2}\geq 0.
		\end{align*}
	\end{proof}     
	
	The upcoming lemma is a version of Lemma \ref{lemmaexchangeablepairDVZ} in the setting of Hilbert-valued random variables. 
	\begin{lemma}
		\label{lemmaexchangeablepair}
		Let $X=\sum_{1\leq q\leq N}F_q$, where $F_q\in
		\mathcal{H}^q(K)$ with covariance operator
		$S_q$ and $\E{\norm{F_q}_K^4}<\infty$. It holds that
		
		\begin{enumerate}[label=(\alph*)]
			\item $\lim_{t\to 0}\frac{1}{t}\E{\inner{F_q^t-F_q,Dg(X)}_K}=-q\E{\inner{ F_q,Dg(X)}_K}.$
			\item   $\lim_{t\to 0}\frac{1}{t}\E{\norm{F^t_q-F_q}^2_K}= 2q\E{\norm{F_q}^2_K}.$
			\item $\lim_{t\to
				0}\frac{1}{2t}\E{\inner{-L^{-1}\brac{F_q^t-F_q},D^2g(X)(F_p^t-F_p)}_K}\\  \hspace*{16.5em}=\frac{1}{q}\sum_{i,j\in\N}\E{\widetilde{\Gamma}(F_{q,i},F_{p,j})\inner{k_i,D^2g(X)k_j }_K}.$
			\item
			$
			\lim_{t\to 0}\frac{1}{t}\E{\norm{F^t_q-F_q}_K^4}=
			4\sum_{i,j\in \N}\E{F_{q,i}^2\brac{\widetilde{\Gamma}(F_{q,j},F_{q,j})-q\E{F_{q,j}}}}\\\hspace*{11.5em}+8\sum_{i,j\in \N}\E{F_{q,i}F_{q,j}\brac{\widetilde{\Gamma}(F_{q,i},F_{q,j})-q\E{F_{q,i}F_{q,j}}}}\\
			\hspace*{11.5em}-4q\sum_{i,j\in \N}\brac{\E{F_{q,i}^2F_{q,j}^2}-\E{F_{q,i}^2}\E{F_{q,j}^2}-2\E{F_{q,i}F_{q,j}}^2}.
			$
		\end{enumerate}
		In particular, when $q=p$ then part (c) becomes
		\begin{align*}
			\lim_{t\to 0}\frac{1}{2t}\E{\inner{-L^{-1}\brac{F_q^t-F_q},D^2g(X)(F_q^t-F_q)}_K}=\Tr_K\brac{D^2g(X)\Gamma \brac{F_q,-L^{-1}F_q}}.
		\end{align*}
	\end{lemma}

	\begin{proof}
		Part $(a)$ follows from
		\begin{align*}
			\lim_{t\to 0}\frac{1}{t}\E{\inner{F_q^t-F_q,Dg(X)}_K}
			&=\lim_{t\to 0}\frac{1}{t}\sum_{i\in \N}\E{\inner{\brac{F^t_{q,i}-F_{q,i}} k_i,Dg(X)}_K}\\
			&=\sum_{i\in\N}\E{\lim_{t\to 0}\frac{1}{t}\E{F^t_{q,i}-F_{q,i}|\eta}{\inner{ k_i,Dg(X)}_K}}\\
			&=-q\sum_{i\in\N}\E{F_{q,i}{\inner{ k_i,Dg(X)}_K}}\\
			&=-q\E{\inner{F_q,Dg(X)}_K}.
		\end{align*}
		
		Part $(b)$ is a result of
		\begin{align*}
			\E{\lim_{t\to 0}\frac{1}{t}\E{\norm{F^t_q-F_q}^2_K|\eta}}
			=&\E{\sum_{i\in\N}\lim_{t\to 0}\frac{1}{t}\E{\brac{F^t_{q,i}-F_{q,i}}^2|\eta}}\\
			=&2\sum_{i\in \N}\E{\widetilde{\Gamma}\brac{F_{q,i},F_{q,i}}}\\ =&2q\E{\norm{F_q}^2_K}.
		\end{align*}
		For part $(c)$, we can write
		\begin{align*}
			&\lim_{t\to 0}\frac{1}{2t}\E{\inner{-L^{-1}\brac{F_q^t-F_q},D^2g(X)(F_p^t-F_p)}_K}\\
			&\hspace*{11em}=\lim_{t\to 0}\frac{1}{2t}\E{\inner{\sum_{i\in \N}\frac{1}{q}\brac{F^t_{q,i}-F_{q,i}}k_i,D^2g(X)\sum_{j\in \N}\brac{F^t_{p,j}-F_{p,j}}k_j}_K}\\
			&\hspace*{11em}=\frac{1}{q}\sum_{i,j\in \N}\E{\lim_{t\to 0}\frac{1}{2t}\E{\brac{F^t_{q,i}-F_{q,i}}\brac{F^t_{p,j}-F_{p,j}}|\eta} \inner{k_i,D^2g(X)k_j}_K}\\
			&\hspace*{11em}=\frac{1}{q}\sum_{i,j\in \N}\E{\widetilde{\Gamma}(F_{q,i},F_{p,j})\inner{k_i,D^2g(X)k_j}_K}.
		\end{align*}
		Using the above expression in the case $q=p$, along
		with the fact that
		\begin{align*}
			\Gamma \brac{F_q,F_q}k_j=\Gamma\brac{\sum_{i\in \N} F_{q,i}k_i,\sum_{m\in \N} F_{q,m}k_m}k_j&=\sum_{i,m\in \N}\Gamma\brac{F_{q,i}k_i,F_{q,m}k_m}k_j\\
			&=\sum_{i,m\in \N}\frac{1}{2}\widetilde{\Gamma}\brac{F_{q,i},F_{q,m}}\brac{k_i\otimes k_m+k_m\otimes k_i}k_j\\
			&=\sum_{i\in \N}\widetilde{\Gamma}\brac{F_{q,i},F_{q,j}}k_i
		\end{align*}
		yields
		\begin{align*}   
			\lim_{t\to 0}\frac{1}{2t}\E{\inner{-L^{-1}\brac{F_q^t-F_q},D^2g(X)(F_q^t-F_q)}_K}=\Tr_K\brac{D^2g(X)\Gamma \brac{F_q,-L^{-1}F_q}}.
		\end{align*}
		For part $(d)$, the exchangeability of $\brac{F_q,F_q^t}$ and Lemma \ref{lemma_limit_realvalued4thpower} imply
		\begin{align*}
			\lim_{t\to 0}\frac{1}{t}\E{\norm{F^t_q-F_q}_K^4}=&\lim_{t\to 0}\frac{1}{t}\E{\norm{\sum_{i\in\N}\brac{F_{q,i}^t-F_{q,i}}k_i}^4_K}\nonumber\\
			=&\lim_{t\to 0}\frac{1}{t}\E{\sum_{i,j\in \N}\brac{F_{q,i}^t-F_{q,i}}^2\brac{F_{q,j}^t-F_{q,j}}^2}\nonumber\\
			=&4\sum_{i,j\in
				\N}\E{F_{q,i}^2\brac{\widetilde{\Gamma}(F_{q,j},F_{q,j})-q\E{F_{q,j}^2}}}\nonumber
			\\
			&+8\sum_{i,j\in \N}\E{F_{q,i}F_{q,j}\brac{\widetilde{\Gamma}(F_{q,i},F_{q,j})-q\E{F_{q,i}F_{q,j}}}}\nonumber\\
			&-4q\sum_{i,j\in \N}\brac{\E{F_{q,i}^2F_{q,j}^2}-\E{F_{q,i}^2}\E{F_{q,j}^2}-2\E{F_{q,i}F_{q,j}}^2}.
		\end{align*}

	\end{proof}    
	%%%%%%%%%%%%%%%%%%%%%%%%%%%%%%%%%%%%%%%%%%%%%%%%%%%%
	
	Via Lemma \ref{lemmaexchangeablepair}, we will establish a fourth moment bound that will help with the remainder term in the proof of Theorem \ref{theorem_fourmomentHilbert}.
	
	\begin{lemma}
		\label{lemma_remaindertermbound}
		Let $X=\sum_{1\leq q\leq N}F_q$, where $F_q\in
		\mathcal{H}^q(K)$ with covariance operator
		$S_q$ and $\E{\norm{F_q}_K^4}<\infty$. It holds that 
		\begin{align*}
			&\lim_{t\to 0}\frac{1}{t} \E{\norm{-L^{-1}\brac{X^t-X}}_K\norm{X^t-X}^2_K}\\&\leq 2N\sqrt{\max_{1\leq p\leq N}\E{\norm{F_p}^2_K}}\sum_{1\leq q\leq N} \sqrt{4q-3}\sqrt{{\norm{F_q}_K^4-\E{\norm{F_q}^2_K}^2-2\norm{S_q}^2_{\operatorname{HS}}}}.
		\end{align*}
	\end{lemma}
	\begin{proof}
		
		% \lim_{t\to 0}\frac{1}{t} \norm{F^t_q-F_q}_K
		We have
		\begin{align*}
			\E{\norm{-L^{-1}\brac{X^t-X}}_K\norm{X^t-X}^2_K}
			&\leq \E{\brac{\sum_{1\leq q\leq N}\frac{1}{q}\norm{F^t_q-F_q}_K}\brac{\sum_{p=1}^N \norm{F^t_p-F_p}_K}^2}\\
			&\leq N\sum_{1\leq p,q\leq N}\frac{1}{q}\E{\norm{F^t_q-F_q}_K\norm{F^t_p-F_p}_K^2}\\
			&\leq N\sum_{1\leq p,q\leq N} \frac{1}{q}\sqrt{\E{\norm{F^t_q-F_q}_K^2}}\sqrt{\E{\norm{F^t_p-F_p}_K^4}}.
		\end{align*}
		The first line is due to chaos decomposition of $X,X^t$ and triangle inequality of $\norm{\cdot}_K$. The second line is a consequence of $\brac{\sum_{p=1}^N y_p}^2\leq N\sum_{i=1}^Ny_p^2$. The last line is due to H\"{o}lder's inequality.   
		
		Next, part (b) of Lemma \ref{lemmaexchangeablepair} implies that
		\begin{align}
			\label{estimate_rightsidecontainlimitF^4}
			&\lim_{t\to 0}\frac{1}{t} \E{\norm{-L^{-1}\brac{X^t-X}}_K\norm{X^t-X}^2_K}\nonumber\\
			&=N\sum_{1\leq p,q\leq N} \frac{1}{q}\sqrt{\lim_{t\to 0}\frac{1}{t}\E{\norm{F^t_q-F_q}_K^2}}\sqrt{\lim_{t\to 0}\frac{1}{t}\E{\norm{F^t_p-F_p}_K^4}}\nonumber\\
			&\leq  N\sqrt{2\max_{1\leq q\leq N}\E{\norm{F_q}^2_K}}\sum_{1\leq q\leq N}\sqrt{\lim_{t\to 0}\frac{1}{t}\E{\norm{F^t_q-F_q}_K^4}}.
		\end{align}
		Let us study the remaining limit on the right hand side. Lemma \ref{lemmaexchangeablepair} says 
		\begin{align}
			\label{limit_4thpowerqthchaos}
			\lim_{t\to 0}\frac{1}{t}\E{\norm{F^t_q-F_q}_K^4}
			=&4\sum_{i,j\in
				\N}\E{F_{q,i}^2\brac{\widetilde{\Gamma}(F_{q,j},F_{q,j})-q\E{F_{q,j}^2}}}\nonumber
			\\
			&+8\sum_{i,j\in \N}\E{F_{q,i}F_{q,j}\brac{\widetilde{\Gamma}(F_{q,i},F_{q,j})-q\E{F_{q,i}F_{q,j}}}}\nonumber\\
			&-4q\sum_{i,j\in \N}\brac{\E{F_{q,i}^2F_{q,j}^2}-\E{F_{q,i}^2}\E{F_{q,j}^2}-2\E{F_{q,i}F_{q,j}}^2}.
		\end{align}
		We will treat each term of
		\eqref{limit_4thpowerqthchaos} separately. For the
		first term of \eqref{limit_4thpowerqthchaos}, our
		proof will use an argument similar to the proof of \cite[Lemma 2.2]{dobler-vidotto-zheng:2018:fourth-moment-theorems} or \cite[Lemma 3.1]{dobler-peccati:2018:fourth-moment-theorem}. First, observe that if $k$ is a fixed positive integer and $J_k$ denotes the projection into the $k$-th Poisson chaos, then
		\begin{align*}
			\E{J_k\brac{\norm{F_p}^2_K}^2}=\sum_{i,j\in\N}\E{J_k\brac{F^2_{q,i}}J_k\brac{F^2_{q,j}}}.
		\end{align*}
		In particular, the expansion in \cite[Lemma 5.1]{dobler-vidotto-zheng:2018:fourth-moment-theorems} yields
		\begin{align*}
			\E{J_{2q}\brac{\norm{F_q}_K^{\otimes 2}}^2}&=\sum_{i,j\in\N}(2q)!\inner{f_{q,i}\widetilde{\otimes} f_{q,i},f_{q,j}\widetilde{\otimes} f_{q,j}}_{\mathfrak{H}^{2q}}\\&=\sum_{i,j\in\N}\brac{2\E{F_{q,i}F_{q,j}}^2+\sum_{r=1}^{q-1}q!^2{q\choose r}^2\inner{f_{q,i}{\star}_r^r f_{q,j},f_{q,j}{\star}_r^r f_{q,i}}_{\mathfrak{H}^{2q-2r}}}.
		\end{align*}
		Thus, the first term of \eqref{limit_4thpowerqthchaos} can be bounded via
		\begin{align*}
			\sum_{i,j\in \N}\E{F_{q,i}^2\brac{\widetilde{\Gamma}(F_{q,j},F_{q,j})-q\E{F_{q,j}^2}}}
			\leq& \frac{1}{2}\sum_{i,j\in \N}\sum_{k=1}^{2q-1}(2p-k)\E{J_k\brac{F_{q,i}^2}J_k\brac{F_{q,j}^2}}\\
			=&\frac{1}{2}\sum_{k=1}^{2q-1}(2p-k)\E{J_k\brac{\norm{F_q}^2_K}^2}\\
			\leq& \frac{2q-1}{2}\sum_{k=1}^{2q-1}\E{J_k\brac{\norm{F_q}^2_K}^2}\\
			=& \frac{2q-1}{2}\brac{\norm{F_q}_K^4-\E{\norm{F_q}^2_K}^2-2\norm{S_q}^2_{\operatorname{HS}}}\\
			&-\frac{2q-1}{2}\sum_{i,j\in \N}\sum_{r=1}^{q-1}q!^2{q\choose r}^2\inner{f_{q,i}{\star}_r^r f_{q,j},f_{q,j}{\star}_r^r f_{q,i}}_{\mathfrak{H}^{2q-2r}}.
		\end{align*}
		
		\noindent     The second term of
		\eqref{limit_4thpowerqthchaos} will receive a similar
		treatment. Based on \cite[Lemma 5.1]{dobler-vidotto-zheng:2018:fourth-moment-theorems}, we have
		\begin{align*}
			\E{J_{2q}(F_{q,i}F_{q,j})}&=\sum_{i,j\in\N}(2q)!\norm{f_{q,i}\widetilde{\otimes} f_{q,j}}^2_{\mathfrak{H}^{2q}}\\
			&=\sum_{i,j\in\N}\brac{\E{F_{q,i}F_{q,j}}^2+\E{F_{q,i}^2}\E{F_{q,j}^2}+\sum_{r=1}^q q!^2{q\choose r}^2\norm{f_{q,i}{\star}_r^r f_{q,j}}^2_{\mathfrak{H}^{2q-2r}}}.
		\end{align*}
		Hence,
		\begin{align*}
			\sum_{i,j\in \N}\E{F_{q,i}F_{q,j}\brac{\widetilde{\Gamma}(F_{q,i},F_{q,j})-q\E{F_{q,i}F_{q,j}}}}
			=&\frac{1}{2}\sum_{i,j\in \N}\sum_{k=1}^{2q-1}(2q-k)\E{J_k\brac{F_{q,i}F_{q,j}}^2}\\
			\leq& \frac{2q-1}{2}\sum_{i,j\in \N}\sum_{k=1}^{2q-1}\E{J_k\brac{F_{q,i}F_{q,j}}^2}\\
			=& \frac{2q-1}{2}\brac{\E{F_{q,i}^2F_{q,j}^2}-\E{F_{q,i}^2}\E{F_{q,j}^2}-2\E{F_{q,i}F_{q,j}}^2}\\
			&-\frac{2q-1}{2}\sum_{i,j\in \N}\sum_{r=1}^{q-1} q!^2{q\choose r}^2\norm{f_{q,i}{\star}_r^r f_{q,j}}^2_{\mathfrak{H}^{2q-2r}}
			\\
			\leq& \frac{2q-1}{2}\brac{\norm{F_q}_K^4-\E{\norm{F_q}^2_K}^2-2\norm{S_q}^2_{\operatorname{HS}}}\\
			&-\frac{2q-1}{4}\sum_{i,j\in \N}\sum_{r=1}^{q-1} q!^2{q\choose r}^2\norm{f_{q,i}{\star}_r^r f_{q,j}}^2_{\mathfrak{H}^{2q-2r}}.
		\end{align*}
		In addition, based on that fact that
		\begin{align*}
			&\sum_{i,j\in\N}\brac{2\norm{f_{q,i}{\star}_r^r f_{q,j}}^2_{\mathfrak{H}^{2q-2r}}+2\inner{f_{q,i}{\star}_r^r f_{q,j},f_{q,j}{\star}_r^r f_{q,i}}_{\mathfrak{H}^{2q-2r}}}\\
			& \qquad\qquad=\sum_{i,j\in\N}\brac{\norm{f_{q,i}{\star}_r^r f_{q,j}}^2_{\mathfrak{H}^{2q-2r}}+2\inner{f_{q,i}{\star}_r^r f_{q,j},f_{q,j}{\star}_r^r f_{q,i}}_{\mathfrak{H}^{2q-2r}}+\norm{f_{q,j}{\star}_r^r f_{q,i}}^2_{\mathfrak{H}^{2q-2r}}}\\
			&\qquad\qquad
			=\sum_{i,j\in\N}\brac{\norm{f_{q,i}{\star}_r^r
					f_{q,j}+f_{q,j}{\star}_r^r
					f_{q,i}}^2_{\mathfrak{H}^{2q-2r}}}\\
			&\qquad\qquad \geq 0,
		\end{align*}
		we deduce from \eqref{limit_4thpowerqthchaos} that
		\begin{align*}
			\lim_{t\to 0}\frac{1}{t}\E{\norm{F^t_q-F_q}^4_K}
			\leq  (8q-6)\brac{\norm{F_q}_K^4-\E{\norm{F_q}^2_K}^2-2\norm{S_q}^2_{\operatorname{HS}}}.
		\end{align*}
		Combining this with \eqref{estimate_rightsidecontainlimitF^4}, we arrive at the fourth moment bound in the statement of this lemma. 
	\end{proof}
	%%%%%%%%%%%%%%%%%%%%%%%%%%%%%%%%%%%%%%%%%%%%%%%%%%%%
	
	The result below is an adaptation to our setting of a classical
	combinatorial identity appearing in \cite[Proof of Proposition 11.2.2]{peccati-taqqu:2011:wiener-chaos-moments}.
	\begin{lemma}
		\label{lemma_contraction_00}
		The quantity $\norm{f_{q,i}\widetilde{\star}^0_0
			f_{p,j}}^2_{\mathfrak{H}^{q+p}}$ appearing in
		Equation \eqref{fourthmoment} can be written in terms
		of norms of non-symmetrized contractions as
		\begin{align*}
			(q+p)!\norm{f_{q,i}\widetilde{\star}^0_0 f_{p,j}}^2_{\mathfrak{H}^{q+p}}=&\bigg(q!p!\norm{f_{q,i}}^2_{\mathfrak{H}^q}\norm{f_{p,j}}^2_{\mathfrak{H}^p}+q!^2\inner{f_{q,i},f_{q,j}}^2_{\mathfrak{H}^q}\mathds{1}_{\left\{ q=p \right\}}
			\\
			&+q!p!{q\choose q\wedge p}{p\choose q\wedge
				p}\norm{f_{q,i} \star^{q\wedge p}_{q\wedge p}
				f_{p,j}}^2_{\mathfrak{H}^{\otimes \abs{q-p}}}
			\mathds{1}_{\left\{ q\neq p \right\}}\\
			&+\sum_{r=1}^{q\wedge p -1}q!p!{q\choose r}{p\choose r} \norm{f_{q,i} \star^r_r f_{p,j}}^2_{\mathfrak{H}^{q+p-2r}}\bigg).
		\end{align*}
	\end{lemma}
	\begin{proof}
		The procedure in \cite[Proof of Proposition 11.2.2]{peccati-taqqu:2011:wiener-chaos-moments} will be slightly modified to fit our situation. Let $\mathfrak{S}_{q+p}$ be the sets of all permutations of $(q+p)$ elements and assume $\pi,\rho\in \mathfrak{S}_{q+p}$. When the intersection set $\{\pi(1),\ldots,\pi(q) \}\cap \{\rho(q+1),\ldots,\rho(q+p) \}$ contains $r$ element, this will be denoted by $\pi \stackrel{r}{\sim}\rho$. Since $\mathfrak{H}=L^2(\mathcal{Z},\mu)$, we have that
		\begin{align}
			\label{contraction00}
			\norm{f_{q,i}\widetilde{\star}^0_0 f_{p,j}}^2_{\mathfrak{H}^{q+p}}=&\norm{{f_{q,i}\widetilde{\otimes} f_{p,j}}}^2_{\mathfrak{H}^{q+p}}\nonumber\\
			=&\frac{1}{(q+p)!^2}\sum_{\pi,\rho\in \mathfrak{S}_{q+p}}\int_{\mathcal{Z}^{q+p}}f_{q,i}\brac{z_{\pi(1)},\ldots,z_{\pi(q)}}f_{p,j}\brac{z_{\pi(q+1)},\ldots,z_{\pi(q+p)}}\nonumber\\
			&\qquad\qquad\qquad f_{q,i}\brac{z_{\rho(1)},\ldots,z_{\rho(q)}}f_{p,j}\brac{z_{\rho(q+1)},\ldots,z_{\rho(q+p)}}\mu(dz_1\ldots dz_{q+p})\nonumber\\
			=&\frac{1}{(q+p)!^2}\sum_{\pi\in \mathfrak{S}_{q+p}}\brac{\sum_{r=1}^{q\wedge p -1}\sum_{\pi\stackrel{r}{\sim} \rho}A_{1,r}+\sum_{\pi \stackrel{0}{\sim}\rho}A_2+\sum_{\pi \stackrel{q\wedge p}{\sim}\rho}A_3}.
		\end{align}
		For the second sum in \eqref{contraction00}, $\pi
		\stackrel{0}{\sim}\rho$ is equivalent to
		
		\begin{equation*}
			\begin{cases}
				\{\pi(1),\ldots,\pi(q) \}\cap \{\rho(1),\ldots,\rho(q) \}=\{\pi(1),\ldots,\pi(q) \}\\
				\{\pi(q+1),\ldots,\pi(q+p) \}\cap \{\rho(q+1),\ldots,\rho(q+p) \}=\{\pi(q+1),\ldots,\pi(q+p) \}
			\end{cases},
		\end{equation*}
		which implies that
		\begin{align*}
			A_2&=\int_{\mathcal{Z}^{q\wedge p}}\brac{\int_{\mathcal{Z}^{q\wedge p}}f_{q,i}\brac{z_{\pi(1)},\ldots,z_{\pi(q)}}f_{q,i}\brac{z_{\pi(1)},\ldots,z_{\pi(q)}}}\\&\brac{f_{p,j}\brac{z_{\pi(q+1)},\ldots,z_{\pi(q+p)}}f_{p,j}\brac{z_{\pi(q+1)},\ldots,z_{\pi(q+p)}}}\mu(dz_1\ldots dz_{q+p})
			=\norm{f_{q,i}}^2_{\mathfrak{H}^q}\norm{f_{p,j}}^2_{\mathfrak{H}^p}.
		\end{align*}
		Furthermore, observe that for a fixed element $\pi\in
		\mathfrak{S}_{q+p}$, there are $q!$ ways to permute
		$\{1,\ldots,q \}$ and $p!$ ways to permute
		$\{q+1,\ldots,q+p \}$. Since $f_{q,i}$ and $f_{p,j}$ are symmetric functions, we have
		\begin{align*}
			\sum_{\pi \stackrel{0}{\sim}\rho}A_2=q!p! \norm{f_{q,i}}^2_{\mathfrak{H}^q}\norm{f_{p,j}}^2_{\mathfrak{H}^p}.
		\end{align*}
		For the third sum in \eqref{contraction00}, there are
		two cases to consider. If $q=p$ then $\pi
		\stackrel{q}{\sim}\rho$ means
		\begin{equation*}
			\begin{cases}
				\{\pi(1),\ldots,\pi(q) \}\cap \{\rho(q+1),\ldots,\rho(2q) \}= \{\pi(1),\ldots,\pi(q) \}
				\\
				\{\pi(q+1),\ldots,\pi(2q) \}\cap \{\rho(1),\ldots,\rho(q) \}= \{\pi(q+1),\ldots,\pi(2q) \}
			\end{cases},
		\end{equation*}
		which implies that
		\begin{align*}
			A_3=&
			\int_{\mathcal{Z}^{q}}\brac{\int_{Z^{q}}f_{q,i}\brac{z_{\pi(1)},\ldots,z_{\pi(q)}}f_{q,j}\brac{z_{\pi(1)},\ldots,z_{\pi(q)}}}\\
			&\qquad\qquad\qquad\qquad\qquad\quad f_{q,i}\brac{z_{\pi(q+1)},\ldots,z_{\pi(2q)}}f_{q,j}\brac{z_{\pi(q+1)},\ldots,z_{\pi(2q)}}\mu(dz_1\ldots dz_{2q})\\
			=&\inner{f_{q,i},f_{q,j}}^2_{\mathfrak{H}^q}\mathds{1}_{\left\{ q=p \right\}},
		\end{align*}
		and there are $q!^2$ copies like the one above. On the other hand for $q\neq p$,
		\begin{align*}
			A_3&=\int_{\mathcal{Z}^{\abs{q-p}}}\brac{\int_{\mathcal{Z}^{q\wedge p}} f_{q,i}\brac{z_{\pi(1)},\ldots,z_{\pi(q)}}f_{p,j}\brac{z_{\rho(q+1)},\ldots,z_{\rho(q+p)}}}\\
			&\qquad\qquad\qquad\qquad \brac{\int_{\mathcal{Z}^{q\wedge p}}f_{q,i}\brac{z_{\rho(1)},\ldots,z_{\rho(q)}}f_{p,j}\brac{z_{\pi(q+1)},\ldots,z_{\pi(q+p)}}}\mu(dz_1\ldots dz_{q+p})\\
			&=\int_{\mathcal{Z}^{\abs{q-p}}}\brac{f_{q,i} \star^{q\wedge p}_{q\wedge p} f_{p,j}}^2\mu(dz_1 \ldots dz_{\abs{q-p}})\\
			&=\norm{f_{q,i} \star^{q\wedge p}_{q\wedge p} f_{p,j}}^2_{\mathfrak{H}^{\otimes \abs{q-p}}}.
		\end{align*}
		Given a fixed $\pi$ such that $\pi
		\stackrel{q\wedge p}{\sim}\rho$ and $q\neq p$, there is a total of ${q\choose
			q\wedge p}{p\choose q\wedge p}$ ways of choosing $q\wedge p$ elements in $ \{\pi(1),\ldots,\pi(q) \}\cap \{\rho(q+1),\ldots,\rho(q+p) \}$ and $q\wedge p$ elements in $\{\pi(q+1),\ldots,\pi(q+p) \}\cap \{\rho(1),\ldots,\rho(q) \}$. In addition, there are $q!p!$ ways to organize $\{\rho(1),\ldots,\rho(q) \}$ and $\{\rho(q+1),\ldots,\rho(q+p) \}$. Therefore, combining the case $q=p$ and $q\neq p$ gives us
		\begin{align*}
			\sum_{\pi \stackrel{q\wedge
					p}{\sim}\rho}A_3=q!^2\inner{f_{q,i},f_{q,j}}^2_{\mathfrak{H}^q}\mathds{1}_{\left\{ q=p \right\}}+q!p!{q\choose
				q\wedge p}{p\choose q\wedge p}\norm{f_{q,i}
				\star^{q\wedge p}_{q\wedge p}
				f_{p,j}}^2_{\mathfrak{H}^{\otimes
					\abs{q-p}}}\mathds{1}_{\left\{ q\neq p \right\}}.
		\end{align*}
		We now turn to the first sum on the right side of
		\eqref{contraction00}, that is when $\pi
		\stackrel{r}{\sim}\rho$ for $1\leq r\leq q\wedge p
		-1$. We can write       
		\begin{align*}
			A_{1,r}&=\int_{\mathcal{Z}^{q+p-2r}}\brac{\int_{Z^r}f_{q,i}\brac{z_{\pi(1)},\ldots,z_{\pi(q)}}f_{p,j}\brac{z_{\rho(q+1)},\ldots,z_{\rho(q+p)}}}\\
			&\qquad\qquad\qquad\qquad \brac{\int_{\mathcal{Z}^r}f_{q,i}\brac{z_{\rho(1)},\ldots,z_{\rho(q)}}f_{p,j}\brac{z_{\pi(q+1)},\ldots,z_{\pi(q+p)}}}\mu(dz_1\ldots dz_{q+p})\\
			&=\int_{\mathcal{Z}^{q+p-2r}}\brac{f_{q,i} \star^r_r f_{p,j}(z_1,\ldots, z_{q+p-2r})}^2\mu(dz_1\ldots dz_{q+p-2r})\\
			&=\norm{f_{q,i} \star^r_r f_{p,j}}^2_{\mathfrak{H}^{q+p-2r}}.
		\end{align*}
		There are ${q\choose r}{p\choose r}$ ways to choose $r$ elements in $ \{\pi(1),\ldots,\pi(q) \}\cap \{\rho(q+1),\ldots,\rho(q+p) \}$ and $r$ elements in $\{\pi(q+1),\ldots,\pi(q+p) \}\cap \{\rho(1),\ldots,\rho(q) \}$. Furthermore, there are $q!p!$ ways to organize $\{\rho(1),\ldots,\rho(q) \}$ and $\{\rho(q+1),\ldots,\rho(q+p) \}$. This yields
		\begin{align*}
			\sum_{r=1}^{q\wedge p -1}\sum_{\pi\stackrel{r}{\sim} \rho}A_{1,r}=\sum_{r=1}^{q\wedge p -1}q!p!{q\choose r}{p\choose r} \norm{f_{q,i} \star^r_r f_{p,j}}^2_{\mathfrak{H}^{q+p-2r}}.
		\end{align*}
		Thus, we can expand \eqref{contraction00} as
		\begin{align*}
			\norm{f_{q,i}\widetilde{\star}^0_0 f_{p,j}}^2_{\mathfrak{H}^{q+p}}=&\frac{(q+p)!}{(q+p)!^2}
			\bigg(q!p!\norm{f_{q,i}}^2_{\mathfrak{H}^q}\norm{f_{p,j}}^2_{\mathfrak{H}^p}+q!^2\inner{f_{q,i},f_{q,j}}^2_{\mathfrak{H}^q}\mathds{1}_{\left\{ q=p \right\}}
			\\
			&+q!p!{q\choose q\wedge p}{p\choose q\wedge
				p}\norm{f_{q,i} \star^{q\wedge p}_{q\wedge p}
				f_{p,j}}^2_{\mathfrak{H}^{\otimes \abs{q-p}}}
			\mathds{1}_{\left\{ q\neq p \right\}}\\
			&+\sum_{r=1}^{q\wedge p -1}q!p!{q\choose r}{p\choose r} \norm{f_{q,i} \star^r_r f_{p,j}}^2_{\mathfrak{H}^{q+p-2r}}\bigg),
		\end{align*}
		which is the desired statement.
	\end{proof}

	\subsection{Lemmas related to the proof of Theorem
		\ref{theorem_Besov}}
	Our first lemma expresses the Hilbert-Schmidt norm in a
	Besov-Liouville space as an norm in $L^2 \left( [0,1]^{\otimes 2} \right)$.
	\begin{lemma}
		\label{lemma_covariance}
		Let $K=\mathcal{I}_{\beta,2}$ and $S$ be the covariance operator of a random variable $X\in L^2(\Omega)\otimes K$. Let $f\in K$, then
		\begin{align}
			\label{fractional_S}
			\brac{D^\beta_{0^{+}}Sf}(s)=\int_0^1\E{\brac{D^\beta_{0^{+}}X}(r) \brac{D^\beta_{0^{+}}X}(s)} \brac{D^\beta_{0^{+}}f}(r)dr
		\end{align}
		is in $L^2([0,1])$. This leads to 
		\begin{align}
			\label{norm_HS}
			\norm{S}_{\operatorname{HS}(K)}=\norm{\E{\brac{D^\beta_{0^{+}}X}(r) \brac{D^\beta_{0^{+}}X}(s)}}_{L^2([0,1]^{2})}.
		\end{align}
	\end{lemma}
	\begin{proof}
		Let $f,g\in K$. Applying Fubini's theorem to
		$\inner{Sf,g}_K=\E{\inner{f,X}_K\inner{g,X}_K}$ and
		rearranging terms yields
		\begin{align*}
			&\int_0^1 \brac{D^\beta_{0^{+}}Sf}(s) \brac{D^\beta_{0^{+}}g}(s)ds\\
			&\qquad\qquad\qquad =\int_0^1 \brac{\int_0^1\E{\brac{D^\beta_{0^{+}}X}(r) \brac{D^\beta_{0^{+}}X}(s)} \brac{D^\beta_{0^{+}}f}(r)dr}\brac{D^\beta_{0^{+}}g}(s)ds,
		\end{align*}
		which is equivalent to
		\begin{align}
			\label{pre_19}
			\int_0^1 \brac{\brac{D^\beta_{0^{+}}Sf}(s)-\int_0^1\E{\brac{D^\beta_{0^{+}}X}(r) \brac{D^\beta_{0^{+}}X}(s)}\brac{D^\beta_{0^{+}}f}(r)dr} \brac{D^\beta_{0^{+}}g}(s)ds=0.
		\end{align}
		Let $\left\{ g_n \right\}_{n \in \N}$ be an orthonormal basis of
		$\mathcal{I}_{\beta,2}$. Due to the isometry between
		$\mathcal{I}_{\beta,2}$ and $L^2([0,1])$, the set
		$\left\{ D^\beta_{0^{+}}g_n \right\}_{n \in \N}$ is an orthonormal basis of
		$L^2([0,1])$. Then, Equation \eqref{pre_19} implies \eqref{fractional_S}.

		To prove \eqref{norm_HS}, let $\{e_n\}_{n\in \N}$ be
		an orthonormal basis of $L^2([0,1])$. Then,
		$\{e_m\otimes e_n\}_{m,n\in \N}$ is an orthonormal
		basis of $L^2([0,1]^{\otimes 2})$. Also,
		$\left\{I^\beta_{0^{+}}e_n  \right\}_{n\in \N}$ is a
		basis of $K$. Now observe that, using
		\eqref{fractional_S}, we can write
		\begin{align*}
			\inner{I^\beta_{0^{+}}e_m,S I^\beta_{0^{+}}e_n}_K&=\int_0^1 e_m(s) 
			\brac{D^\beta_{0^{+}} S I^\beta_{0^{+}}e_n}(s)ds\\
			&=\int_0^1 e_m(s) 
			\brac{\int_0^1\E{\brac{D^\beta_{0^{+}}X}(r) \brac{D^\beta_{0^{+}}X}(s)} e_n(r)dr}ds\\
			&=\int_0^1\int_0^1 
			\E{\brac{D^\beta_{0^{+}}X}(r) \brac{D^\beta_{0^{+}}X}(s)} e_m(s) e_n(r)drds,
		\end{align*}
		which leads to
		\begin{align*}
			\norm{S}_{\operatorname{HS}(K)}^2=\sum_{m,n\in\N} \inner{I^\beta_{0^{+}}e_m,S I^\beta_{0^{+}}e_n}^2_K=\norm{\E{\brac{D^\beta_{0^{+}}X}(r) \brac{D^\beta_{0^{+}}X}(s)}}_{L^2([0,1]^{\otimes 2})}^2,
		\end{align*}
		where the first equality comes from the identity
		$\norm{T}^2_{\operatorname{HS}(K)}=\sum_{m,n\in\N}
		\inner{k_m,T k_n}^2_K$ for an operator $T \in \operatorname{HS}(K)$ and an orthonormal basis $\{k_n\}_{n\in\N}$ of $K$.
	\end{proof}
	\begin{remark}
		Let $\zeta$ be a $L^2([0,1])$-valued random variable
		with covariance operator $T$. Note that the second statement in Lemma \ref{lemma_covariance}
		is comparable to the identity             
		\begin{equation*}
			\norm{T}_{\operatorname{HS}(L^2([0,1]))}=\norm{\E{\zeta(r)\zeta(s)}}_{L^2([0,1]^{\otimes 2})}
		\end{equation*}
		whenever $T \in
		\operatorname{HS}\left( L^2\left([0,1]\right)
		\right)$.
	\end{remark}
	The following lemma is helpful to compute the Hilbert-Schmidt norms of
	the Poisson process and Brownian motion appearing in Subsection \ref{subsection_Besov}.
	\begin{lemma}
		\label{lemma_cov_kernel}
		Let the setting of Subsection \ref{subsection_Besov}
		prevail, where $X_\lambda$ and $Z$ denoting a Poisson
		process and a Brownian motion in
		$\mathcal{I}_{\beta,2}$, respectively. Then, one has
		\begin{align*}
				\E{\brac{D^\beta_{0^{+}}Z}(r) \brac{D^\beta_{0^{+}}Z}(s)}&=\E{\brac{D^\beta_{0^{+}}X_\lambda}(r) \brac{D^\beta_{0^{+}}X_\lambda}(s)}\\&=\frac{1}{\Gamma(-\beta+1)^2}\int_0^{r\wedge s} (r-x)^{-\beta}(s-x)^{-\beta}dx.
		\end{align*}
	\end{lemma}     
	\begin{proof}
		According to \cite[Section 3.1]{coutin-decreusefond:2013:steins-method-brownian}, the covariance
		operator of our Brownian motion is $S'=I^\beta_{0^{+}}
		I^{1-\beta}_{0^{+}} I^{1-\beta}_{1^{-}}
		D^\beta_{0^{+}}$. Substituting this into Equation
		\eqref{fractional_S}, we get
		\begin{align}
			\label{equation_cov_BM}
			\brac{D^\beta_{0^{+}}I^\beta_{0^{+}} I^{1-\beta}_{0^{+}} I^{1-\beta}_{1^{-}} D^\beta_{0^{+}}f}(s)=\int_0^1\E{\brac{D^\beta_{0^{+}}Z}(r) \brac{D^\beta_{0^{+}}Z}(s)} \brac{D^\beta_{0^{+}}f}(r)dr.
		\end{align}
		For the left-hand side, note that $f\in 
		\mathcal{I}_{\beta,2}$ implies that $
		D^\beta_{0^{+}}f\in L^2\subseteq L^1$, so that $I^{1-\beta}_{0^{+}} I^{1-\beta}_{1^{-}} D^\beta_{0^{+}}f \in L^1$. Thus, $D^\beta_{0^{+}}I^\beta_{0^{+}}=I$ by \cite[Theorem 2.4]{samko-kilbas-marichev:1993:fractional-integrals-derivatives}.
		Continuing with the left-hand side, we first write out
		$I^{1-\beta}_{0^{+}}$ using its definition and then
		perform an integration by part, which yields
		\begin{align*}
			\brac{ I^{1-\beta}_{0^{+}} I^{1-\beta}_{1^{-}} D^\beta_{0^{+}}f}(s)=&\frac{1}{\Gamma({1-\beta})}\int_0^1 1_{[0,s]}(r) (s-r)^{-\beta}\brac{I^{1-\beta}_{1^{-}} D^\beta_{0^{+}}f}(r) dr\\
			=&\frac{1}{\Gamma({1-\beta})}\int_0^1 I^{1-\beta}_{0^{+}} \brac{1_{[0,s]}(\cdot) (s-\cdot)^{-\beta}}(r)\brac{ D^\beta_{0^{+}}f}(r) dr
		\end{align*}
		In particular, the integration by part is valid since \cite[Equation (2.20)]{samko-kilbas-marichev:1993:fractional-integrals-derivatives} is satisfied for $p=q=2$ and $0<\beta<1/2$. Equation \eqref{equation_cov_BM} then becomes
		\begin{align*}
			\int_0^1\brac{\frac{1}{\Gamma({1-\beta})}I^{1-\beta}_{0^{+}} \brac{1_{[0,s]}(\cdot) (s-\cdot)^{-\beta}}(r)-\E{\brac{D^\beta_{0^{+}}Z}(r) \brac{D^\beta_{0^{+}}Z}(s)}} \brac{D^\beta_{0^{+}}f}(r)dr=0.
		\end{align*}
		Now, using a basis argument like the one in the proof of Lemma \ref{lemma_covariance} yields
		\begin{align*}
			\E{\brac{D^\beta_{0^{+}}Z}(r) \brac{D^\beta_{0^{+}}Z}(s)}&=\frac{1}{\Gamma({1-\beta})}I^{1-\beta}_{0^{+}} \brac{1_{[0,s]}(\cdot) (s-\cdot)^{-\beta}}(r)\\
			&=\frac{1}{\Gamma({1-\beta})^2}\int_0^r (r-x)^{-\beta}(s-x)^{-\beta}1_{[0,s]}(x)dx\\
			&=\frac{1}{\Gamma({1-\beta})^2}\int_0^{r\wedge s} (r-x)^{-\beta}(s-x)^{-\beta}dx.
		\end{align*}
		
		Next, we will compute $\E{\brac{D^\beta_{0^{+}}X_\lambda}(r)
			\brac{D^\beta_{0^{+}}X_\lambda}(s)}$. Recall the representation of
		$X_\lambda$ given at \eqref{def_poiprocess_Besov}. In
		order to use this representation, we need the
		joint density of $(T_n,T_m)$. By definition, $T_{m
			\wedge n}$ and $T_{m \vee n}-T_{m \wedge n}$ are
		independent and distributed as $\Gamma({m \wedge
			n},\lambda)$ and $\Gamma(\abs{m-n},\lambda)$,
		respectively. Their joint density is hence given by
		\begin{align*}
			f_{T_{m \wedge n},T_{m \vee n}-T_{m \wedge n}}(x,y)=\frac{\lambda^{m\vee n}}{\Gamma(n\vee m)\Gamma(\abs{m-n})}x^{n\wedge m-1}y^{\abs{m-n}-1}e^{-\lambda(x+y)}.
		\end{align*}
		Since $T_{m \vee n}=T_{m \wedge n}+(T_{m \vee n}-T_{m
			\wedge n})$, we can write, using a simple change of variable,
		\begin{align}
			\label{formula_density}
			f_{T_{m \wedge n},T_{m \vee n}}(x,y)= \frac{\lambda^{m\vee n}}{\Gamma(n\wedge m)\Gamma(\abs{m-n})}x^{n\wedge m-1}(y-x)^{\abs{m-n}-1}e^{-\lambda y} \mathds{1}_{\{x<y\}}. 
		\end{align}
		We are now ready to compute $\E{\brac{D^\beta_{0^{+}}X_\lambda}(r) \brac{D^\beta_{0^{+}}X_\lambda}(s)}$. We have
		\begin{align}
			\label{kernel_X_expand}
			&\E{\brac{D^\beta_{0^{+}}X_\lambda}(r)
				\brac{D^\beta_{0^{+}}X_\lambda}(s)}\nonumber \\
			&\qquad\qquad =\frac{1}{\lambda
				\Gamma(-\beta+1)^2}\Bigg(\sum_{n,m\in
				\N}
			\E{(r-T_n)^{-\beta}_{+}
				(s-T_m)^{-\beta}_{+}}
			-\frac{\lambda s^{ -\beta+1}}{-\beta+1}\sum_{n\in
				\N} \E{(r-T_n)^{-\beta}_{+} }\nonumber \\
			&\qquad\qquad \quad-\frac{\lambda r^{
					-\beta+1}}{-\beta+1}\sum_{n\in \N}
			\E{(s-T_m)^{-\beta}_{+} }+\frac{\lambda^2}{(-\beta+1)^2}s^{-\beta+1}r^{-\beta+1}\Bigg)\nonumber\\
			&\qquad\qquad =\frac{1}{\lambda
				\Gamma(-\beta+1)^2}\Bigg(\sum_{n\in \N}
			\E{(r-T_n)^{-\beta}_{+}
				(s-T_n)^{-\beta}_{+}}+\sum_{n\in
				\N} \sum_{m\neq n}\E{(r-T_n)^{-\beta}_{+}
				(s-T_m)^{-\beta}_{+}}\nonumber \\
			&\qquad\qquad \quad-\frac{\lambda }{-\beta+1}s^{ -\beta+1}\sum_{n\in
				\N} \E{(r-T_n)^{-\beta}_{+}
			}-\frac{\lambda }{-\beta+1}r^{
				-\beta+1}\sum_{n\in \N} \E{(s-T_m)^{-\beta}_{+}
			}\nonumber \\
			&\qquad\qquad \quad+\frac{\lambda^2}{(-\beta+1)^2}s^{-\beta+1}r^{-\beta+1}\Bigg).
		\end{align}
		The first sum on the right side (consisting of all diagonal terms when $m=n$) simplifies as  
		\begin{align*}
			&\frac{1}{\lambda \Gamma(-\beta+1)^2}\sum_{n\in \N} \E{(t-T_n)^{-\beta}_{+} (s-T_n)^{-\beta}_{+}}\\
			&\qquad\qquad\qquad =\frac{1}{\lambda \Gamma(-\beta+1)^2}\sum_{n\in \N}\int_0^\infty (r-x)^{-\beta}_{+} (s-x)^{-\beta}_{+}\frac{\lambda^n}{\Gamma(n)}x^{n-1}e^{-\lambda x} dx\\
			&\qquad\qquad\qquad=\frac{1}{\Gamma(-\beta+1)^2}\int_0^{r\wedge s}(r-x)^{-\beta} (s-x)^{-\beta}e^{-\lambda x} \brac{\sum_{n\in \N} \frac{(\lambda x)^{n-1}}{(n-1)!}} dx\\
			&\qquad\qquad\qquad=\frac{1}{\Gamma(-\beta+1)^2}\int_0^{r\wedge s}(r-x)^{-\beta} (s-x)^{-\beta}  dx.
		\end{align*}
		Next, we consider the second sum on the right side of
		\eqref{kernel_X_expand}. The joint density of
		$(T_n,T_m)$ given in \eqref{formula_density} enables us to write
		%%%%%%%%%%%%%%%%%%%%%%
		\begin{align*}
				&\sum_{n\in \N} \sum_{m\neq n}\E{(r-T_n)^{-\beta}_{+} (s-T_m)^{-\beta}_{+}}\\&=\sum_{n\in \N} \sum_{m=n+1}^\infty\E{(r-T_n)^{-\beta}_{+} (s-T_m)^{-\beta}_{+}}+\sum_{m\in \N} \sum_{m+1}^\infty\E{(r-T_n)^{-\beta}_{+} (s-T_m)^{-\beta}_{+}}\\
				&=\lambda^2\int_0^{r\wedge s}\int_0^s(r-x)^{-\beta} (s-y)^{-\beta}e^{-\lambda y}\sum_{n\in \N} \sum_{m=n+1}^\infty\frac{(\lambda x)^{n-1}(\lambda y-\lambda x)^{m-n-1}}{\Gamma(n)\Gamma(m-n)}dydx\\
				&+\lambda^2\int_0^{r\wedge s}\int_0^r(s-x)^{-\beta} (r-y)^{-\beta}e^{-\lambda y}\sum_{m\in \N} \sum_{n=m+1}^\infty\frac{(\lambda x)^{m-1}(\lambda y-\lambda x)^{n-m-1}}{\Gamma(m)\Gamma(n-m)}dydx
			\end{align*}
			By letting $k=m-n$, it is easy to see that 
			\begin{align*}
				\sum_{n\in \N}  \sum_{m=n+1}^\infty\frac{(\lambda x)^{n-1}(\lambda y-\lambda x)^{m-n-1}}{\Gamma(n)\Gamma(m-n)}=e^{\lambda y},
			\end{align*}
			and hence
			\begin{align*}
				&\sum_{n\in \N} \sum_{m\neq n}\E{(r-T_n)^{-\beta}_{+} (s-T_m)^{-\beta}_{+}}\\
				&=\lambda^2\int_0^{r\wedge s}\int_0^s (r-x)^{-\beta} (s-y)^{-\beta}dydx+\lambda^2\int_0^{r\wedge s}\int_0^r(s-x)^{-\beta} (r-y)^{-\beta}dydx\\
				&=\frac{\lambda^2}{-\beta+1}\int_0^{r\wedge s}(r-x)^{-\beta} (s-x)^{-\beta}(s+t-2x)  dx\\
				&=\frac{\lambda^2}{(-\beta+1)^2}(s-x)^{-\beta+1}(r-x)^{-\beta+1}\Big|_0^{r\wedge s} \\
				&=\frac{\lambda^2}{(-\beta+1)^2}\brac{s^{-\beta+1}r^{-\beta+1}-(s-r\wedge s)^{-\beta+1}(r-r\wedge s)^{-\beta+1}}\\
				&=\frac{\lambda^2}{(-\beta+1)^2}s^{-\beta+1}r^{-\beta+1}. 
		\end{align*}
		
		%%%%%%%%%%%%%%%%%%%%
		
		For the remaining sums in \eqref{kernel_X_expand}, observe that 
		\begin{align*}
			\sum_{m\in\N}\E{(s-T_m)^{-\beta}_{+} }=\frac{\lambda}{-\beta+1}s^{-\beta+1}.
		\end{align*}
		Substituting the previous calculations into
		\eqref{kernel_X_expand} yields
		\begin{align*}
			\E{\brac{D^\beta_{0^{+}}X_\lambda}(r) \brac{D^\beta_{0^{+}}X_\lambda}(s)}=\frac{1}{\Gamma(-\beta+1)^2}\int_0^{r\wedge s} (r-x)^{-\beta}(s-x)^{-\beta}dx.
		\end{align*}

	\end{proof}

	\bibliography{refs}
\end{document}